\documentclass[%
 aip,
 amsmath,amssymb,
reprint,%
]{revtex4-1}
\usepackage{graphicx}
\usepackage{dcolumn}
\usepackage{bm}
\usepackage{amsthm}
\usepackage[utf8]{inputenc}
\usepackage[T1]{fontenc}
\usepackage{mathptmx}
\usepackage{etoolbox}
\usepackage{bbm}
\usepackage{xcolor}
\usepackage[caption=false]{subfig}

\newtheorem{theorem}{Theorem}

\newtheorem{prop}{Proposition}

\makeatletter
\def\@email#1#2{%
 \endgroup
 \patchcmd{\titleblock@produce}
  {\frontmatter@RRAPformat}
  {\frontmatter@RRAPformat{\produce@RRAP{*#1\href{mailto:#2}{#2}}}\frontmatter@RRAPformat}
  {}{}
}%
\makeatother
\begin{document}

\preprint{AIP/123-QED}

\title[Rate and Noise-Induced Tipping Working in Concert]{Rate and Noise-Induced Tipping Working in Concert}
\author{Katherine Slyman}%
 \email{kslyman@ad.unc.edu.}
\affiliation{ 
Department of Mathematics, University of North Carolina at Chapel Hill, Chapel Hill, North Carolina 27517
}%
\author{Christopher K. Jones}
\affiliation{Renaissance Computing Institute, University of North Carolina at Chapel Hill, Chapel Hill, North Carolina 27517}

\date{\today}

\begin{abstract}
Rate-induced tipping occurs when a ramp parameter changes rapidly enough to cause the system to tip between co-existing, attracting states. We show that the addition of noise to the system can cause it to tip well below the critical rate at which rate-induced tipping would occur. Moreover it does so with significantly increased probability over the noise acting alone. We achieve this by finding a global minimizer in a canonical problem of the Freidlin-Wentzell action functional of large deviation theory that represents the most probable path for tipping. This is realized as a heteroclinic connection for the Euler-Lagrange system associated with the Freidlin-Wentzell action and we find it exists for all rates less than or equal to the critical rate. Its role as most probable path is corroborated by direct Monte Carlo simulations.
\end{abstract}

\maketitle

\begin{quotation}
The IPCC \cite{collins_m_m_sutherland_l_bouwer_s-m_cheong_t_frolicher_h_jacot_des_combes_m_koll_roxy_i_losada_k_mcinnes_b_ratter_e_rivera-arriaga_rd_susanto_d_swingedouw_and_l_tibig_ocean_2022} defines a tipping point as "a level of change in system properties beyond which a system reorganises, often in a nonlinear manner, and does not return to the initial state even if the drivers of the change are abated." For climate systems, tipping points refer to a critical thresholds when global or regional climate switch stable states \cite{collins_m_m_sutherland_l_bouwer_s-m_cheong_t_frolicher_h_jacot_des_combes_m_koll_roxy_i_losada_k_mcinnes_b_ratter_e_rivera-arriaga_rd_susanto_d_swingedouw_and_l_tibig_ocean_2022}.

Climate change is a rate-induced tipping problem. It is also a noisy system. These two components are true for many relevant climate subsystems. Since conceptual models of these climate systems contain multiple mechanisms that can induce tipping, there is a clear need for mathematical approaches which synthesize techniques from these areas of research.

To begin this analysis, we look at this question in its
most simple form: in the context of a canonical problem.
We study a system with a ramp parameter and impose additive noise on the dynamics to study to what extent the ramp parameter and noise interact to facilitate tipping. We approach the problem using a dynamical systems framework and prove the existence of a heteroclinic orbit between a stable base state and threshold boundary.

We find this heteroclinic orbit corresponds to the most probable path between these points. For rate values less than some critical rate, a ramp parameter alone does not allow tipping. The addition of noise to the system causes tipping well below the critical rate needed for rate-induced tipping to occur. However, noise alone acting on the system induces tipping but only after significantly longer time. Therefore a ramp parameter and noise conspire to cause tipping with increased probability over either acting alone.
\end{quotation}

\section{\label{sec:level1}Introduction}

There are three main mechanisms for tipping in dynamical systems: bifurcation-induced, rate-induced, and noise-induced \cite{ashwin_tipping_2012}. This work focuses on when there is a parameter shift (R-tipping) and the addition of random fluctuations (N-tipping), the schematics of which are shown in Figure \ref{FIG:schematic}. The aim is to assess the extent rate and noise-induced tipping work together to facilitate tipping in cases where neither readily tip on their own.

We consider a canonical one-dimensional system with a ramp parameter and impose additive noise on its dynamics. We find the addition of noise to the system causes the system to tip for all $r$ values less than the critical rate needed for rate-induced tipping, and does so with significantly increased probability over the noise or ramp acting alone. The most probable path to tip for all $r$ values corresponds to the global minimizer of the Freidlin-Wentzell rate functional, which itself is a heteroclinic orbit. While we show these results in context of a canonical problem, the phenomenon we find is suggestive for tipping between stable base states and threshold boundaries.

Our methods are as follows. We compactify the system and derive the Euler-Lagrange equations associated with the Freidlin-Wentzell action functional \cite{freidlin_random_2012}. Using a dynamical systems framework, including tracking invariant manifolds, using the Wazewski principle, and applying shooting methods, we prove there is always an intersection of the unstable manifold of the base state and stable manifold of the threshold state. Through numerical simulations we find that this intersection is unique. The action values indicate that the heteroclinic connection through this intersection point is the global minimizer of the Freidlin-Wentzell functional. The fact that is does correspond to the most probable path at the appropriate noise levels is shown from Monte-Carlo simulations.

As we consider nonzero rates within the ramp parameter, the ramp is a nontrivial component of the system. Consequently, this means that the additive noise should be of small levels, as otherwise the noise effects would come after the ramp finishes, and we focus on the interplay of these phenomena. A drawback of Freidlin-Wentzell theory is that it necessitates vanishingly small noise \cite{freidlin_random_2012}. In focusing on a small but not vanishingly small noise regime, the transient behavior of the underlying deterministic system will play an important role. We find Freidlin-Wentzell theory still holds in regard to the dynamical structure for small noise strengths, namely the heteroclinic connection is the most probable path, but more discussion is needed when considering the expected time to tip. The extension from vanishingly small to small noise levels is relevant for several applications of interest, especially in environmental, social or biological contexts. 

The addition of noise, regardless of size, will always result in tipping of the system in infinite time. However, if we consider a finite time horizon, the addition of noise will cause the system to tip with a certain probability. The probability of tipping is dependent on both the noise strength and the speed of the ramp parameter, $r$, where the time horizon is chosen long enough to ensure the ramp function completes its transition. The size of small noise will change depending on the value of $r$ being considered. Noise strengths are chosen so that the probability of tipping is less than $21\%$. 

\begin{figure}[ht]
    \centering
    \includegraphics[scale=.35]{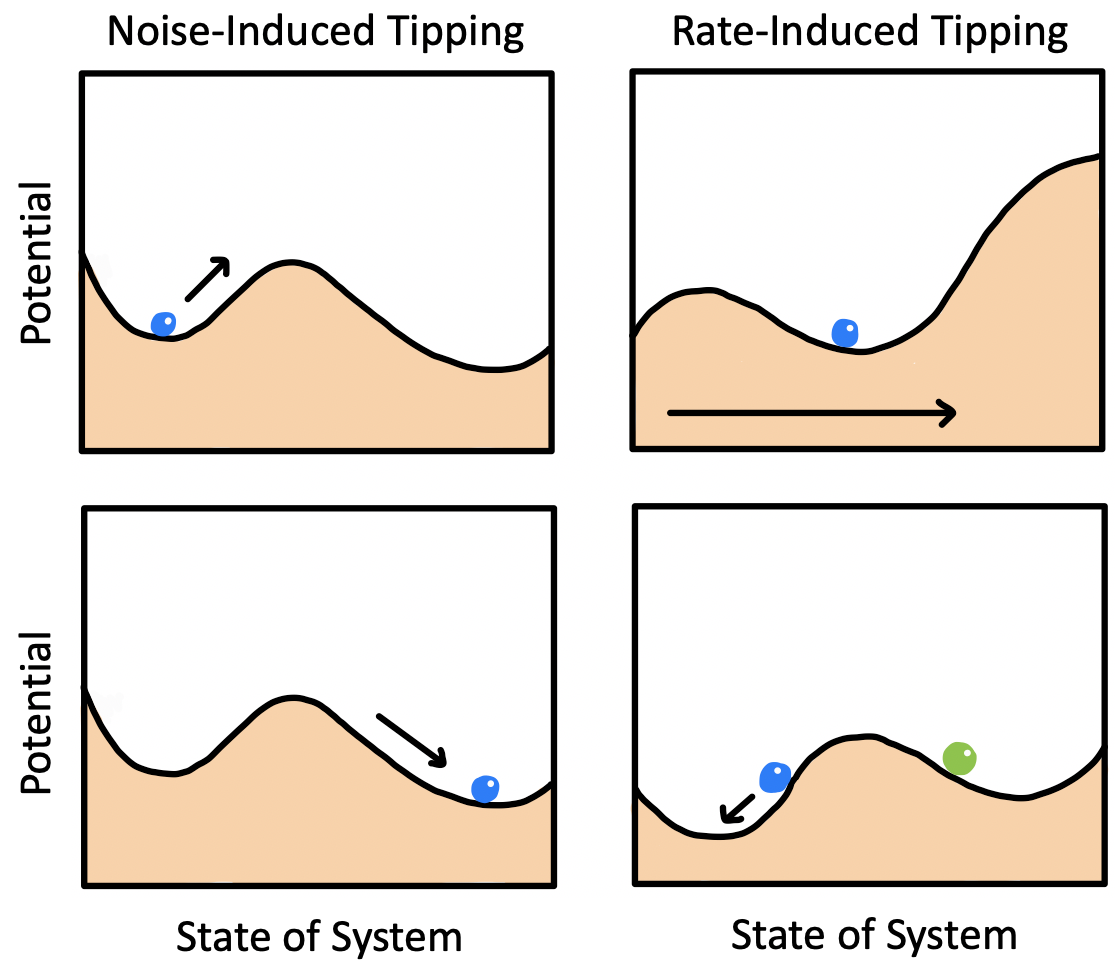}
    \caption{Schematics for noise and rate-induced tipping in terms of a potential function and initializing with a particle at a minimum. N-tipping occurs when a noisy fluctuation is strong enough to push the particle out of the minimum and to some local maximum, where it can then fall to another minimum. R-tipping occurs when an external input varies too fast compared to the response rate of the system, resulting in a shift of the landscape and the deviation of a particle from its initialized stable state and the start of tracking a different stable state. This figure is inspired and recreated from \citet{van_der_bolt_understanding_2021}.
    }
    \label{FIG:schematic}
\end{figure}

Our analysis builds off the work of \citet{ashwin_parameter_2017} and \citet{ritchie_early-warning_2016}. \citet{ashwin_parameter_2017} introduced and studied the prototype model for rate-induced tipping
\begin{equation}
\begin{aligned}
    \dot{x}=(x+\lambda)^2-1,
\end{aligned}
\end{equation}
with a monotonically increasing time-dependent parameter,
\begin{equation}
\begin{aligned}
    \lambda(t)=\frac{\lambda_{max}}{2} \left(1+\tanh \left(\frac{\lambda_{max}rt}{2} \right)\right), \ \ r>0.
\end{aligned}
\end{equation}
Using a compactification \cite{wieczorek_compactification_2021}, they augment the system to an autonomous two-dimensional system containing equilibria and compact invariant sets, and in turn, the rate-induced tipping problem turns into a heteroclinic connection problem between two saddle equilibria. \citet{perryman_how_2015} finds the critical rate needed for tipping within the system is $r_c=4/3$. \citet{ritchie_early-warning_2016} then considered this canonical problem with additive noise, and found that an interplay between the noise and ramp parameter results in tipping of the system before the critical rate, $r_c$, is reached. However, they only consider $r$ values close to the critical rate. \citet{ritchie_early-warning_2016} find solutions of the variational problem determining the most likely tipping path using numerical continuation techniques. The majority of their work focuses on the most likely tipping time in the plane of two parameters: distance from tipping threshold and noise intensity.

The motivation of this work relates to climate subsystems. The Earth’s climate is changing due to steadily warming temperatures caused by rising levels of greenhouse gasses \cite{us_epa_climate_2015}. Moreover, there are parts of the Earth system that have the potential for large, abrupt, and irreversible transitions in response to this warming, and could lead to cascading effects \cite{collins_m_m_sutherland_l_bouwer_s-m_cheong_t_frolicher_h_jacot_des_combes_m_koll_roxy_i_losada_k_mcinnes_b_ratter_e_rivera-arriaga_rd_susanto_d_swingedouw_and_l_tibig_ocean_2022}. These changes can be characterized as tipping points \cite{collins_m_m_sutherland_l_bouwer_s-m_cheong_t_frolicher_h_jacot_des_combes_m_koll_roxy_i_losada_k_mcinnes_b_ratter_e_rivera-arriaga_rd_susanto_d_swingedouw_and_l_tibig_ocean_2022}. As presented in \citet{lenton_early_2011} and \citet{lenton_tipping_2008}, there are many such examples: Greenland ice sheet loss, break-off of Antarctic ice-sheets, boreal forest dieback, and permafrost loss, to name a few. Given the magnitude of the impacts of these phenomena, a comprehensive understanding of tipping phenomena is needed to predict and prevent these irreparable changes. Many conceptual models of climate systems contain multiple mechanisms that can induce tipping and there is a clear need for mathematical approaches which combine techniques from both rate-induced and noise-induced tipping.

The paper is structured as follows. Section \ref{Det} begins with the deterministic dynamics of the canonical problem. In Section \ref{stoch} we build the stochastic framework by introducing additive noise to the system. In Section \ref{Quadratic} we derive and study the most probable path equations. These equations lead to a theorem about the existence of a heteroclinic orbit. Lastly, in Section \ref{comps}, we perform a numerical investigation of this problem that includes the finding the heteroclinic connections, path actions, and probability of tipping for different values of the rate and noise strength. We finish with a discussion and concluding remarks.

\section{The Deterministic Dynamics and Rate-Induced Tipping}
\label{Det}
Rate-induced tipping is where a sufficiently quick change to a parameter of a system may cause the system to move away from one attractor to another \cite{ashwin_tipping_2012}. We consider
\begin{equation}
\begin{aligned}
    \dot{x}=(x+y)^2-1,
\end{aligned}
\label{EQ:xdotOG}
\end{equation}
where $^\cdot =\frac{d}{dt}$, and a monotonically increasing time-dependent parameter, as proposed by  \citet{ashwin_parameter_2017},
\begin{equation}
\begin{aligned}
    y(t)=\frac{3}{2} \left(1+\tanh \left(\frac{3rt}{2} \right)\right), \ \ r>0.
\end{aligned}
\label{EQ:ramp} 
\end{equation}
Reformulating the nonautonomous system in \eqref{EQ:xdotOG} into the compactified system using the ramp function, \eqref{EQ:ramp}, itself as the coordinate transformation, maps the real line onto the finite $y-$interval $(0,3)$. The $y$-interval is closed by including $y=0,3$ which come from the limits of \eqref{EQ:ramp} at $\pm$infinity. This leads to the autonomous compactified two-dimensional system
\begin{equation}
\begin{aligned}
    \dot{x}&=(x+y)^2-1, \\
    \dot{y}&=ry(3-y).
\end{aligned}
\label{EQ:2DCompactifyFinal} 
\end{equation}
\\
The system given in \eqref{EQ:2DCompactifyFinal} has four fixed points. We focus on the saddle equilibria $(-1,0)$ and $(-2,3)$.  At a critical $r$, which we denote $r_c$, there is a heteroclinic connection between the two saddle points. \citet{perryman_how_2015} found that $r_c=4/3$ and the connecting orbit is the line given by $x=-\frac{y}{3}-1$. However, for $r<r_c$, the system end-point tracks the saddle equilibrium initialized at $(-1,0)$ to $(-4,3)$ and when $r>r_c$, the system tips to infinity. In Figure \ref{fig:nonautpathscomp}, we show trajectories for different values of $r$ for the system given in \eqref{EQ:2DCompactifyFinal}, demonstrating solution behaviors when initializing at the saddle $(-1,0)$. 

\begin{figure}[ht]
    \centering
    \includegraphics[scale=.35]{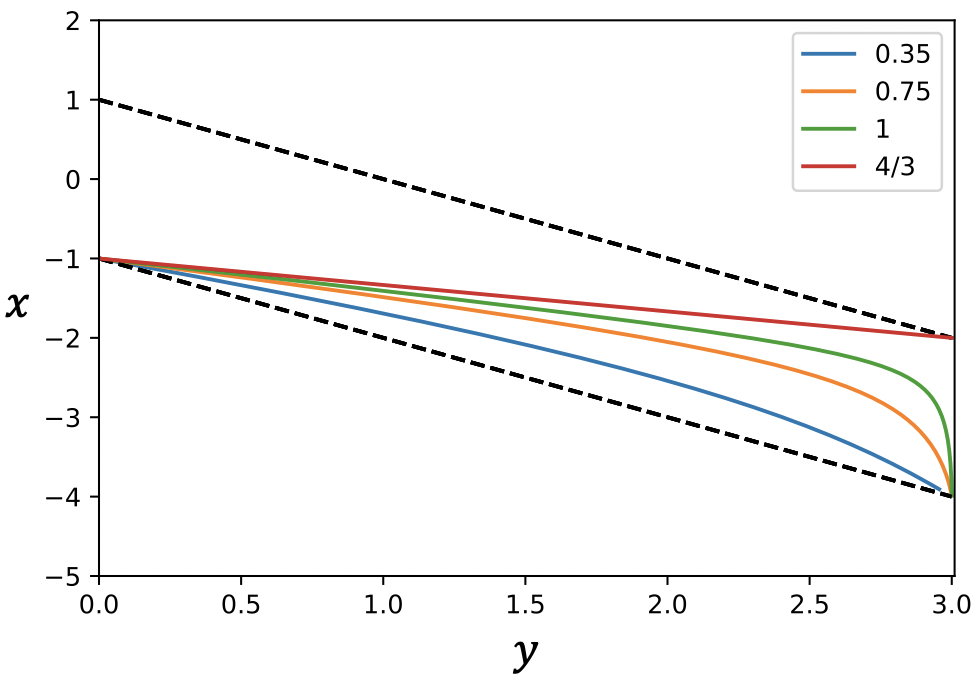}
    \caption{Solutions of the compactified system given by \eqref{EQ:2DCompactifyFinal}. The blacked dash curves track the fixed points $x=1-y$ and $x=-1-y$ in the frozen system over time. The colored trajectories are solution curves initialized at $x=-1,y=2.80729 \times 10^{-13}$ for different values of $r$. Solution curves with $r<4/3$ do not tip, whereas the solution curves with $r\geq 4/3$ tip. There is a heteroclinic connection between the two saddle equilibria at $r=4/3$.}
    \label{fig:nonautpathscomp}
\end{figure}

\section{Building the Stochastic Framework} \label{stoch}
For the remainder of this work, we want to consider the effects of additive noise on the dynamics of $x$ in \eqref{EQ:2DCompactifyFinal}. However, for Freidlin-Wentzell theory, we have to consider noise on both the dynamics of $x$ and $y$ and take the limit as noise goes to zero in the $y$ component. The stochastic version of the canonical problem is of the form
\begin{equation}
\begin{aligned}
    dx&=f(x,y)dt+\sigma_1 dW_1 =((x+y)^2-1)dt+\sigma_1 dW_1, \\
    dy&=g(x,y)dt+\sigma_2 dW_1=(ry(3-y))dt+\sigma_2 dW_2.
\end{aligned}
\label{EQ:generalsde}
\end{equation}
Speaking generally of this form, $x,y$ are stochastic processes parameterized by time, $f,g$ are the deterministic pieces of the system often referred to as the drift, $W_1,W_2$ are standard Wiener processes, and $\sigma_1,\sigma_2$ denote the noise strength and often referred to as the diffusion coefficient.

With the addition of noise to the system, we will have, with probability equal to one, tipping in the system between the two saddle equilibria. We want to find the most probable path to tip between these two points. The tool we use to study these transitions is the Freidlin-Wentzell theory of large deviations. This framework is fully presented in Freidlin and Wentzel's monograph \citep{freidlin_random_2012}, Forgoston and Moore's review article \cite{forgoston_primer_2018} and for gradient systems in Berglund's review article \citep{berglund_kramers_2013}.

As presented in \citet{freidlin_random_2012}, the most probable path between two points $(x_0,y_0)$ and $(x_f,y_f)$ is a curve of the form $(c_1(t),c_2(t))$ that minimizes the Freidlin-Wentzell functional which is given by
\begin{equation}
\begin{aligned}
    I[c_1,c_2]=\int_{t_0}^{t_f} \left(\frac{(\dot{c_1}-f)^2}{\sigma_1^2}+\frac{(\dot{c_2}-g)^2}{\sigma_2^2} \right)dt,
\end{aligned}
\label{EQ:FWfunc}
\end{equation}
where $(c_1(t_0),c_2(t_0))=(x_0,y_0)$ and $(c_1(t_f),c_2(t_f))=(x_f,y_f)$. $I[c_1,c_2]$ is nonnegative and only vanishes if and only if both $\dot{c_1}=f$ and $\dot{c_2}=g$ are solutions to the associated deterministic system. This functional represents the cost of straying from the deterministic dynamics. Minimizing this functional leads to the Euler-Lagrange equations, given by
\begin{equation}
\begin{aligned}
\ddot{c_1}&=f_y \dot{c_2}+ff_x+\frac{\sigma_1^2}{\sigma_2^2}(gg_x-\dot{c_2}g_x) ,  \\
\ddot{c_2}&=g_x \dot{c_1}+gg_y+\frac{\sigma_2^2}{\sigma_1^2}(ff_y-\dot{c_1}f_y) 
\end{aligned}
\label{EQ:finalEL}
\end{equation}
which are a condition critical points, consequently minimizers, of the Freidlin-Wentzell functional must satisfy. These conditions are necessary, but not sufficient for minimizers \cite{freidlin_random_2012}. We will use these conditions to derive the most probable path equations in Section \ref{Quadratic}.

\section{A Dynamical Systems Perspective on the  Canonical Problem}
\label{Quadratic}
Using the Euler-Lagrange equations given by \eqref{EQ:finalEL}, we use a Legendre transform \cite{arnold_mathematical_1997} to create a degree four Hamiltonian system of the form
\begin{equation}
\begin{aligned}
\dot{x}&=f+\sigma_1^2p, \\
\dot{p}&=-f_xp, \\
\dot{y}&=g+\sigma_2^2 q, \\
\dot{q}&=-g_yq.
\end{aligned}
\label{EQ:4D}
\end{equation}
The Hamiltonian function itself is
\begin{equation}H(x,p,y,q)=fp+gq+\frac{\sigma_1^2}{2}p^2+\frac{\sigma_2^2}{2}q^2.
\end{equation}
As mentioned earlier, we want to only consider noise on the dynamics of $x$, as $y$ is a time parameterization, and thus we send $\sigma_2$ to zero. It follows that $\dot{x},\dot{p},\dot{y}$ are all independent of $q$ and we are able to project onto our equations into $x,p,y$ space. Using this independence of $q$ and substituting $f$ and $g$ as they are defined in \eqref{EQ:generalsde} results in \eqref{EQ:4D} becoming
\begin{equation}
\begin{aligned}
\dot{x}&=(x+y)^2-1+\sigma_1^2p, \\
\dot{p}&=-2(x+y)p, \\
\dot{y}&=ry(3-y).
\end{aligned}
\label{EQ:MPPequations}
\end{equation}

In addition, notice that $p=0$ is invariant and carries the determinisitic flow given by \eqref{EQ:2DCompactifyFinal}. These equations in \eqref{EQ:MPPequations} are the most probable path equations. Throughout this work, $x$ is the original state variable, $y$ is a time reparameterization, and $p$ is the extra variable representing the work a trajectory has to do against the vector field. 

We note that alternatively we could have used the Freidlin-Wentzell functional on the nonautonomous system \eqref{EQ:xdotOG} to derive the Euler-Lagrange equations, use a Legendre transform to create a degree two Hamiltonian system, and finish by compactifying the system. The compactification process and the Euler-Lagrange and Legendre transform procedures commute, and we would have the same resulting equations as shown in \eqref{EQ:MPPequations}. This alternative method is useful when we perform numerical experiments in Section \ref{MCSsec}. 

Performing a phase portrait analysis on \eqref{EQ:MPPequations}, we have six equilibria: three on $y=0$ and three on $y=3$. We are interested in the heteroclinic connection between the saddle points $(-1,0,0)$ and $(-2,0,3)$, as these correspond to the saddles $(-1,0)$ and $(-2,3)$ in our two-dimensional phase space. For notation purposes we refer to $(-1,0,0)$ as $s_1$ and $(-2,0,3)$ as $s_2$. A quick check of the eigenvalues of \eqref{EQ:MPPequations} linearized at $s_1$ show $s_1$ has a 1D stable manifold and a 2D unstable manifold. Similar methods show $s_2$ has a 1D unstable manifold and a 2D stable manifold. We denote unstable and stable manifolds of a point $p$ by $W^u(p)$ and $W^s(p)$ respectively. Using this notation, the desired heteroclinic will lie on $W^u(s_1)$ and also on $W^s(s_2)$. See Figure \ref{FIG:3Dphasespace} for what the phase space looks like on $y=0$. We note that asymptotically the phase space dynamics are identical on $y=3$. 
\begin{figure}[ht]
\centering
    \includegraphics[width=.7\linewidth]{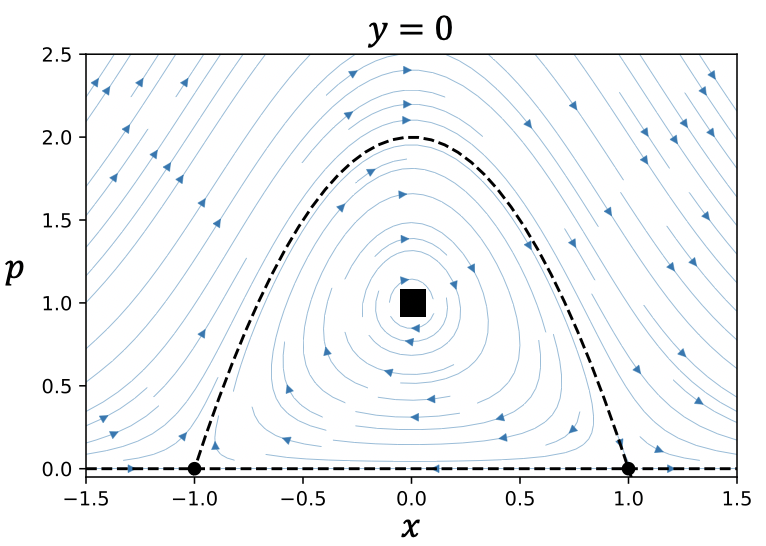}
    \caption{The phase space for \eqref{EQ:MPPequations} on the plane $y=0$. We have two saddles (black circles) and one center (black square). The black dashed lines represent where $H=0$. The blue arrows show the direction of the vector field. The phase space is asymptotically identical on the plane $y=3$.}
    \label{FIG:3Dphasespace}
\end{figure}

Using the Hamiltonian structure in the invariant planes $y=0$ and $y=3$ creates two possible tipping paths between the two saddles of interest. The first possible path is to tip from $s_1$ to $(1,0,0)$ in $y=0$ and then end-point track from $(1,0,0)$ to $s_2$ in $p=0$. The second possible path is to end-point track in $p=0$ from $s_1$ to $(-4,3,0)$ and then tip to $s_2$ in $y=3$. However, as we will see in Section 5.3, these paths have a high action value and have essentially an infinite time until tipping occurs. 

We claim there is always a third heteroclinic connection that is the most probable path and is the path of least action. We first show the existence of a heteroclinic orbit between the two saddles $s_1$ and $s_2$ for all $r \leq r_c$ by showing $W^u(s_1)$ is continuous on the plane $y=-x$ for $y \leq \frac{3}{2}$ for $r \leq r_c$, and that $W^u(s_1)$ and $W^s(s_2)$ are symmetric.
\begin{figure*}
\centering
    \includegraphics[scale=.4]{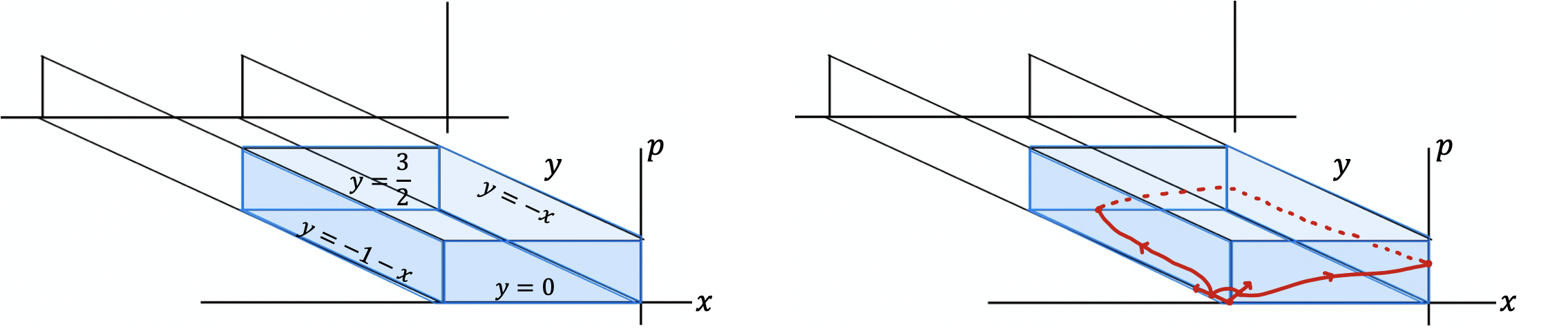}
    \caption{The boundary of the Wazewski set is in blue, and extends infinitely in the positive $p$ direction. It is the space bounded by $y=-1-x, y=0, y=-x, y=\frac{3}{2}$, and $p=0$. Taking a quarter circle of radius $\epsilon$ about $s_1$ intersected with the piece of $W^u(s_1)$ lying in $W$, and applying map $K$, results in the dotted red curve.}
    \label{FIG:wset}
\end{figure*}
\begin{prop}
The primary intersection of $W^u(s_1)$ with the plane $y=-x$ is continuous on the plane $y=-x$ \\
for $y<\frac{3}{2}$.
\end{prop}
\begin{proof}
The Wazewski Principle \cite{srzednicki_wazewski_2004} states the following: 
\

\

\noindent \textit{Let $W^-$ be the immediate exit set of $W$ and let $W^0$ be the eventual exit set of $W$. If $W^-$ is closed relative to $W^0$, then $W$ is a Wazewski set and the map $K:W^0 \rightarrow W^-$, that takes each point to the first where it exits $W$ is continuous.}
\

\

We define the primary intersection of $W^u(s_1)$ with the plane $y=-x$ to be the first crossing of this plane from trajectories initialized in the unstable subspace of $s_1$ coming from $-\infty$.
For the system given in \eqref{EQ:MPPequations}, we say the Wazewski set, $W$, is the space bounded by the following planes: $y=-1-x$, $y=-x$, $y=0$, $y=3/2$, and $p=0$. Based on flow of the vector field, the following are true about the boundaries of $W$: $y=-1-x$ is an entrance set, $y=0$ and $p=0$ are neither entrance nor exit sets, as they are invariant planes, and $y=3/2$ is an exit set. On $y=-x$, below the curve $p=\frac{1}{\sigma_1^2}(1-ry(3-y))$ is an entrance set and above it, an exit set. Refer to Figure \ref{FIG:wset} for a visual of $W$.

We have to determine what happens on the curve $p=\frac{1}{\sigma_1^2}(1-ry(3-y))$ itself, which is the boundary between an entrance set and an immediate exit set.
Consider $x$ and $p$ as functions of $y$. Looking at the first and second derivatives at the point
$z=(-y,\frac{1}{\sigma_1^2}(1-ry(3-y)),y)$, representing any point on this curve, we have
\begin{equation}
\begin{aligned}
\frac{dx}{dy}\Big\rvert_z&=-1, \\
\\
\frac{d^2 x}{dy^2}\Big\rvert_z&=\frac{3-2y}{y(3-y)} >0 \text{ for } y<3/2, \\
\\
\frac{dp}{dy}\Big\rvert_z&=0,\\
\\
\frac{d^2 p}{dy^2}\Big\rvert_z&=0.
\end{aligned}
\label{EQ:wsetarg}
\end{equation}
We see in \eqref{EQ:wsetarg} that $\frac{dx}{dy}=-1$ and $\frac{d^2 x}{dy^2}>0$. By the second derivative test, we know a trajectory would be concave up at this point, forcing any points to leave and consequently, not enter $W$. Therefore we have shown that the boundary of the immediate exit set is contained in the immediate exit set. We conclude the following about the immediate exit set and eventual exit set of $W$:
\begin{align*}
W^-&=\{(x,p,y) \mid y=3/2, y=-x \text{ for } p \geq \frac{1}{\sigma_1^2}(1-ry(3-y))\}\\
W^0&=\{(x,p,y) \mid W \setminus \{y=0, y=-1-x, p=0, (-1,0,0), \\
& \hspace{5mm} (0,1,0)\}\}.
\end{align*}
The boundary of the immediate exit set is in the immediate exit set, and it easily follows that $W^-$ is closed relative to $W^0$. Therefore $W$ is a Wazewski set and the map $K:W^0 \rightarrow W^-$, is continuous for $y<\frac{3}{2}$. This implies that $W^u(s_1)$ intersected with the plane $y=-x$ is continuous for $y<\frac{3}{2}$. 
\end{proof}
\begin{prop}
$W^u(s_1)$ intersected with $W^-$ crosses the plane $y=\frac{3}{2}$ for $r \leq r_c$.
\end{prop}

\begin{proof}
Notice $y=\frac{3}{2}$ separates $W^-$ into two pieces. Take the quarter circle of radius $\epsilon$ around the fixed point $s_1$ intersected with $W^u(s_1)$ that lies in $W$, and call this curve $C_W$. Applying the map $K$ to $C_W$ results in a curve in $\mathbb{R}^3$, specifically a curve lying in $W^-$ by the definition of Wazewski map $K$. 

Since $C_W$ is a closed curve, we track where the two endpoints of $C_W$ map to under $K$. The first endpoint of $C_W$ has $y=0,p\neq 0$ and second endpoint of $C_W$ has $p=0,y\neq 0$. Take the endpoint of $C_W$ that lies in $y=0$. Since the $y=0$ plane is invariant, when we apply $K$, the trajectory must stay in this plane and eventually exit through $y=-x$ and above $p=\frac{1}{\sigma_1^2}(1-ry(3-y))$. Take the endpoint of $C_W$ that lies in $p=0$. Since the $p=0$ plane is invariant, when we apply $K$, the trajectory must stay in this plane. Since $\dot{y}>0$, this trajectory will eventually exit through $y=\frac{3}{2}$, when $r \leq r_c$. 

$W^u(s_1)$ intersected with $W^-$ actually intersects $y=\frac{3}{2}$ by the intermediate value theorem, as $K$ is a continuous map, and one endpoint of $C_W$ maps to the plane $y=-x$ in $y=0$ while the other endpoint of $C_W$ maps to the plane $y=3/2$ in $p=0$. See Figure \ref{FIG:wset} for an illustration of this shooting argument.

Therefore the intersection of $W^u(s_1)$ and the plane $y=-x$ is continuous for $y \leq \frac{3}{2}$ for $r \leq r_c$. 
\end{proof}

\begin{prop}
$W^u(s_1)$ and $W^s(s_2)$ are symmetric. 
\end{prop}

\begin{proof}

Recall our system given in \eqref{EQ:MPPequations}. Making the change of variables $\tau=-t$, we get the time reversed system given by
\begin{equation}
\begin{aligned}
x'&=-(x+y)^2+1-\sigma_1^2 p,\\
p'&=2(x+y)p, \\
y'&=ry(y-3).
\end{aligned}
\label{EQ:timereverse}
\end{equation}
We transform the variables $x,p,y$ by
\begin{equation*}
\hat{x}=-x-3, \qquad
\hat{p}=p,  \qquad 
\hat{y}=3-y.
\end{equation*}
and substitute them into the time reversed system we found in \eqref{EQ:timereverse}. The equations simplify to
\begin{equation}
\begin{aligned}
\hat{x}'&=(\hat{x}+\hat{y})^2-1+\sigma_1^2\hat{p}, \\
\hat{p}'&=-2(\hat{x}+\hat{y})\hat{p}, \\
\hat{y}'&=r\hat{y}(3-\hat{y}).
\end{aligned}
\label{EQ:symmetricsystem}
\end{equation}
We see that \eqref{EQ:symmetricsystem} is in the original form, as given in \eqref{EQ:MPPequations}, and it follows that $W^u(s_1)$ and $W^s(s_2)$ are symmetric. 

\end{proof}

\begin{theorem}
There exists a heteroclinic connection between the saddle points $s_1$ and $s_2$ that goes through the plane $y=-x$ at $y=\frac{3}{2}$ for $r \leq r_c$.
\label{THM1}
\end{theorem}

\begin{proof} 
We found that the intersection of $W^u(s_1)$ and the plane $y=-x$ was continuous for $y \leq \frac{3}{2}$ for $r \leq r_c$ using Propositions 1 and 2. The symmetry of $W^u(s_1)$ and $W^s(s_2)$, proven in Proposition 3, implies the intersection of $W^s(s_2)$ intersected with the plane $y=-x$ is continuous for $y \geq \frac{3}{2}$ for $r \leq r_c$. Therefore $W^u(s_1)$ and $W^s(s_2)$ will always intersect once in the plane $y=-x$ at $y= \frac{3}{2}$, implying a heteroclinic connection between $s_1$ and $s_2$ for $r \leq r_c$, and concluding our proof of Theorem \ref{THM1}.
\end{proof}

We have thus proven the existence of a heteroclinic connection between $s_1$ and $s_2$ for all $r\leq r_c$ and demonstrated how to find this heteroclinic using the intersection of the invariant manifolds. In the deterministic system, for $r<r_c$, we would not have tipping or a heteroclinic connection between the saddles. The presence of noise, regardless the size, allows direct tipping between these saddles within the system.

\section{Computational Methods and Numerical Results} \label{comps}
\subsection{Visualization of Invariant Manifolds and the Heteroclinic Connection}
We proved in Section 4 the existence of the intersection of $W^u(s_1)$ and $W^s(s_2)$ at $y=3/2$, giving rise to a heteroclinic connection between the two saddle points through that specific point. We numerically compute these manifolds, plot them in $y=-x$, and observe their intersection at $y=3/2$. This enables us to visualize their intersection, as well as compute the trajectory through the intersection point. The trajectory is then projected into the $xy$ plane to find the heteroclinic connection in the two-dimensional extended phase space.

The local unstable subspace of $s_1$ is spanned by the two vectors 
\begin{equation}
\left(\frac{\sigma_1^2}{4},1,0 \right)^T \text{ and } \left(\frac{-2}{3r+2},0,1 \right)^T, 
\label{EQ:vectors}
\end{equation}
which span the plane 
\begin{equation}
4(x+1)-(\sigma_1^2+4)p-\frac{12r}{3r+2}y=0.
\label{EQ:planespan}
\end{equation}
Intersecting this plane with the sphere 
\begin{equation}
(x+1)^2+p^2+y^2=(.001)^2
\label{EQ:sphere}
\end{equation} and taking points such that $y,p>0$ result in a curve of points that lie in the unstable subspace, as seen in Figure \ref{fig:disc_curve}. We discretize this curve, and use the tuples as a set of initial conditions.
\begin{figure}[ht]
    \centering
    \includegraphics[scale=.25]{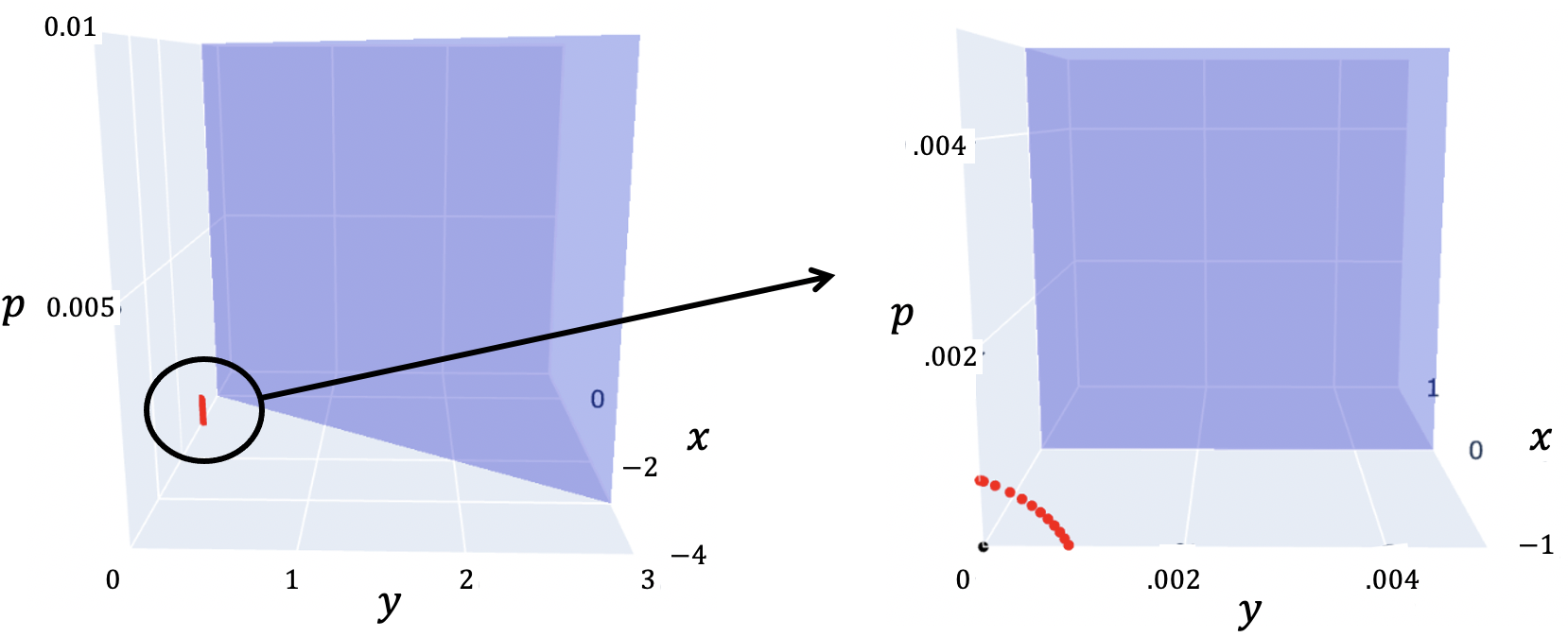}
    \caption{Parameters are set at $r=1, \sigma_1=0.15$. The red curve is the intersection of the sphere \eqref{EQ:sphere} and the plane \eqref{EQ:planespan} spanned by \eqref{EQ:vectors}.}
    \label{fig:disc_curve}
\end{figure}
We numerically run system \eqref{EQ:MPPequations} forward in time, for each initial condition, until the trajectory first hits the plane $y=-x$. Similarly you can perform this process when looking at the stable subspace of $s_2$ and running system \eqref{EQ:MPPequations} backwards in time. We find the intersection of these two curves in the plane $y=-x$. Through these simulations, we find the intersection point of these two curves in $y=-x$ is unique. Running the system both forwards and backwards in time from the intersection point supplies the full heteroclinic trajectory. Refer to Figure \ref{FIG:manifold} to see a visualization of $W^u(s_1)$ and $W^s(s_2)$ intersecting in the plane $y=-x$, as well as the trajectory through the intersection point for two different parameter pairs of $r$ and $\sigma_1$, corresponding to the heteroclinic orbit between $s_1$ and $s_2$. Projecting this heteroclinic orbit into $xy$ space is the connecting orbit between $(-1,0)$ and $(-2,3)$, and we show in the next section that this orbit is in fact the most probable path between these points.
\begin{figure}[ht]
  \centering
  \subfloat[]{\includegraphics[width=0.2\textwidth]{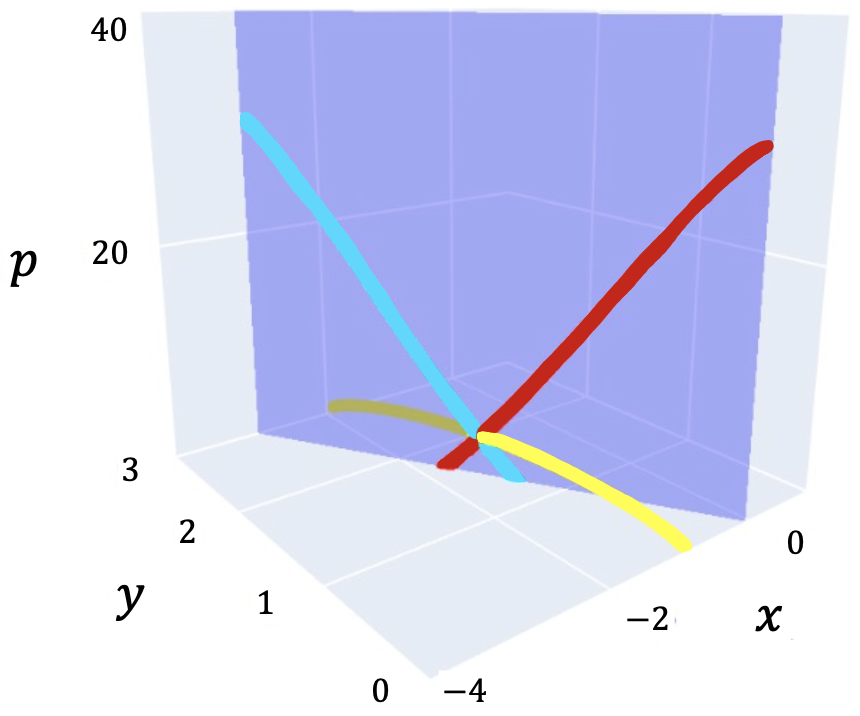}\label{fig:f1}}
  \hspace{5mm}
  \subfloat[]{\includegraphics[width=0.2\textwidth]{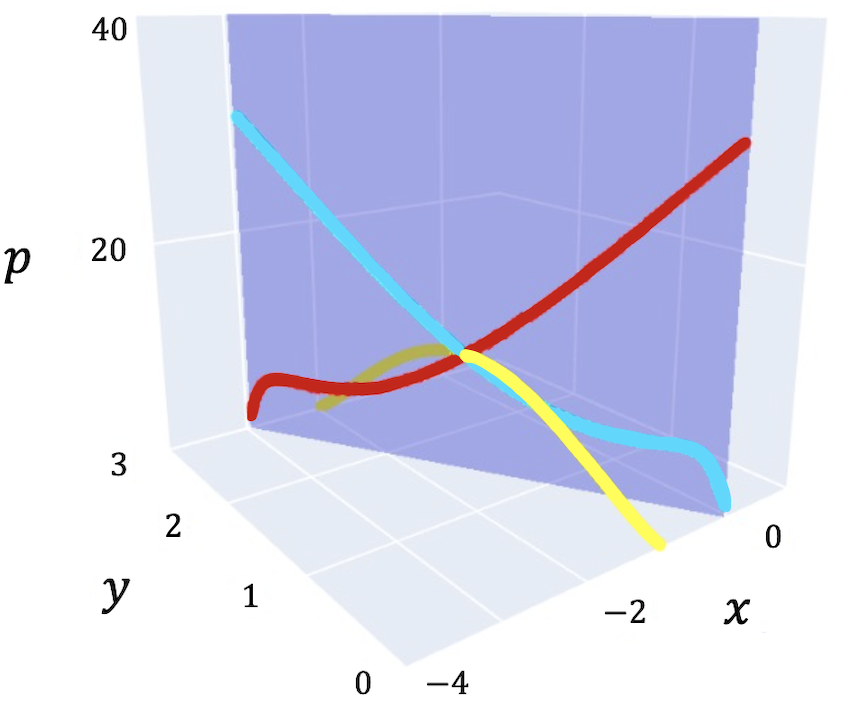}\label{fig:f2}}
  \caption{$W^u(s_1)$ (red) and $W^s(s_2)$ (cyan) in the plane $y=-x$ (purple) for $y \in (0,3)$. The trajectory through their intersection point is the heteroclinic orbit connecting $s_1$ and $s_2$ (yellow). a) Parameters are set at $r=1, \sigma_1=0.25$. b) Parameters are set at $r=.5, \sigma_1=0.25$.}
   \label{FIG:manifold}
\end{figure}

Besides the existence of the heteroclinic connection between $s_1$ and $s_2$ proven in Section \ref{Quadratic}, this first set of numerical simulations now verifies the uniqueness of the heteroclinic connection between $s_1$ and $s_2$. However, we still need to determine if this heteroclinic connection is the most probable path between these two points, implying we need to show it is the global minimizer of the Freidlin-Wentzell action functional. We perform these calculations in the next section.

\subsection{Monte Carlo Simulations and the Most Probable Path} \label{MCSsec}
We corroborate the heteroclinic connection constructed in Section \ref{Quadratic} is in fact the most probable path using Monte Carlo simulations. Recall our original problem was a one-dimensional differential equation. Consider its stochastic version, given by
\begin{equation}
\begin{aligned}
dx&=((x+y)^2-1)dt +\sigma_1 dW, \\
y(t)&=\frac{3}{2} \left (1+\tanh \left(\frac{3rt}{2}\right )\right), r>0.
\end{aligned}   
\label{EQ:MCs}
\end{equation}
As we said in Section \ref{Quadratic}, the order of compactification process and the Euler-Lagrange and Legendre procedures commute, and therefore we use \eqref{EQ:MCs} for running simulations as it is computationally less expensive.

We numerically approximate the solutions of \eqref{EQ:MCs} by using the Euler-Maruyama method to create a discretized Markov process \cite{higham_algorithmic_2001} over the time interval $[0,30].$ To apply the Euler-Maruyama method, we partition the time interval into sub-intervals of width $\Delta t=.001$, and initialize the solution at $x=-1$ and $y=2.80729 \times 10^{-13}$. We note that changing the initial $y$ value, corresponding to changing the starting time to some $t_0=-20,-15,-10,-5$ only shortens or extends the time for a realization to tip \cite{ritchie_early-warning_2016}. To create the discretized Markov process, recursively define $x$ as
\begin{equation}
\begin{aligned}
x_{n+1}&=x_n+((x_n+y_n)^2-1) \Delta t +\sigma_1 \Delta W_n
\end{aligned}  
\label{EQ:EM}
\end{equation}
A standard Weiner process, $W$, satisfies the property that Brownian increments are independent and normally distributed with mean zero and variance $\Delta t$. Therefore it follows that $\Delta W=W_n-W_{n-1}$ can be numerically simulated using $\sqrt{\Delta t} \cdot N(0,1)$. This can be shown by manipulating the probability density function of $N(0,\Delta t)$.

We simulate $N$ realizations of \eqref{EQ:MCs} using the Euler-Maruyama method given in \eqref{EQ:EM}. We map these realizations to two dimensional phase space by plotting $(y(t),x)$. We define tipping to be when a realization of \eqref{EQ:MCs} crosses $W^s(-2,3)$, and $\lim_{t \rightarrow \infty} \neq -4$. $W^s((-2,3)$ can be see in Figure \ref{FIG:manifolds}.

Of the $N$ realizations, we define $M$ to be the number of realizations that tip on the finite time interval of our choosing. Thus $N-M$ do not tip, an example of which is shown in Figure \ref{FIG:MCs}a for $r=1, \sigma_1=0.15$. There are $M$ points within the $M$ realizations that tipped for every discretized time. We use Python to get the kernel smoothing density estimation of the $M$ points at each time. This finds the `most probable point' at every time step, which is determined by the peak of the kernel density estimation. This peak corresponds to the mode of the $M$ points at that time. Plotting the mode at each time step, we have an approximation for the most probable path. Overlaying the numerically simulated most probable path, with what we found using the the projection of the trajectory through intersection of $W^u(s_1)$ and $W^s(s_2)$ in $y=-x$ in Section 5.1, we see that the approximation matches the actual path extremely well, an example of which is shown in Figure \ref{FIG:MCs}b for $r=1, \sigma_1=0.15$. Therefore we can say the trajectories that tipped followed the heteroclinic connection, and that the heteroclinic connection between the two saddles is the most probable path.

\begin{figure}[ht]
\centering
    \includegraphics[scale=.23]{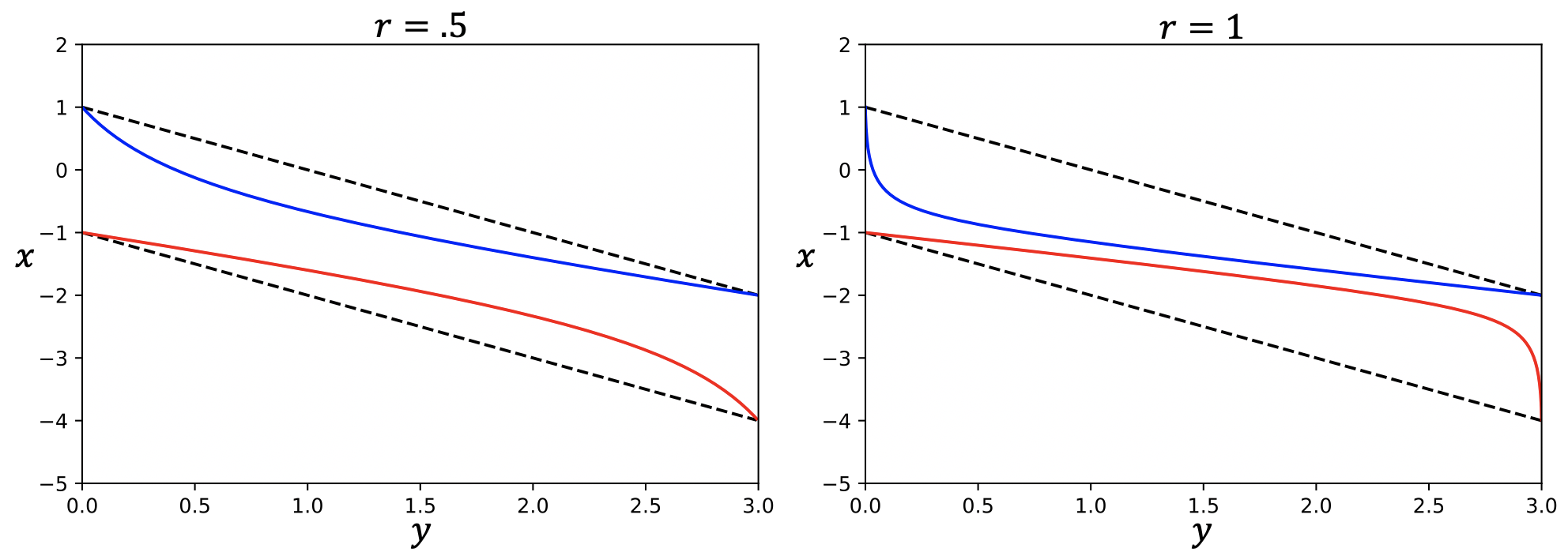}
    \caption{$W^u(-1,0)$ (red) and $W^s(-2,3)$ (blue) for $r=.5$ and $r=1$. Tipping occurs when a realization of \eqref{EQ:EM} crosses $W^s(-2,3)$ and $\lim_{t \rightarrow \infty} \neq -4$.}
    \label{FIG:manifolds}
\end{figure}

\begin{figure*}
  \centering
  \subfloat[]{\includegraphics[width=0.3\textwidth]{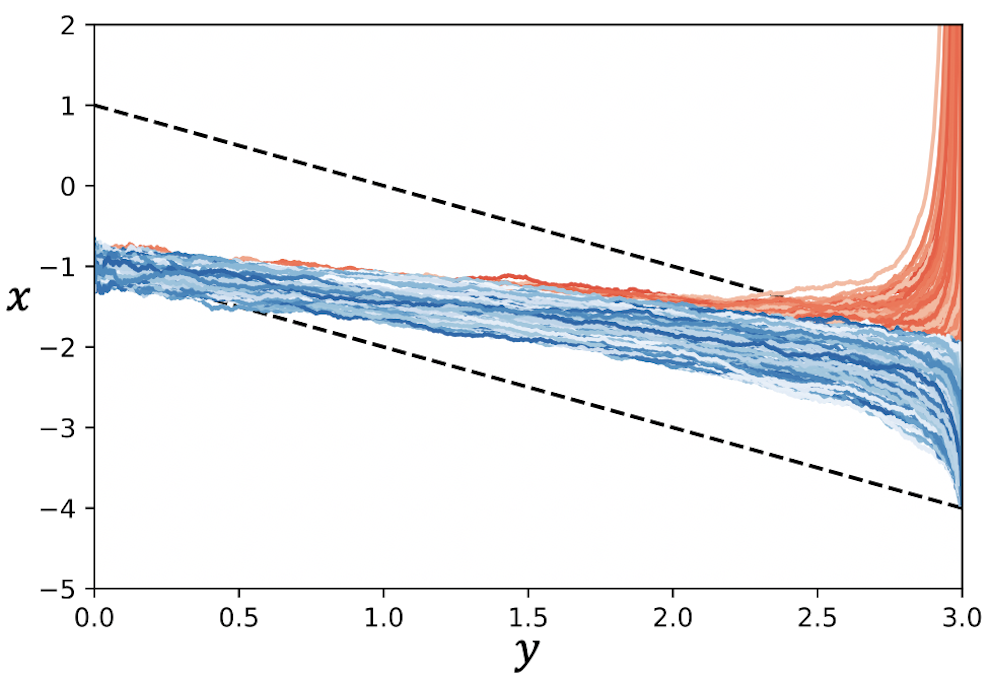}}
  \hspace{5mm}
  \subfloat[]{\includegraphics[width=0.3\textwidth]{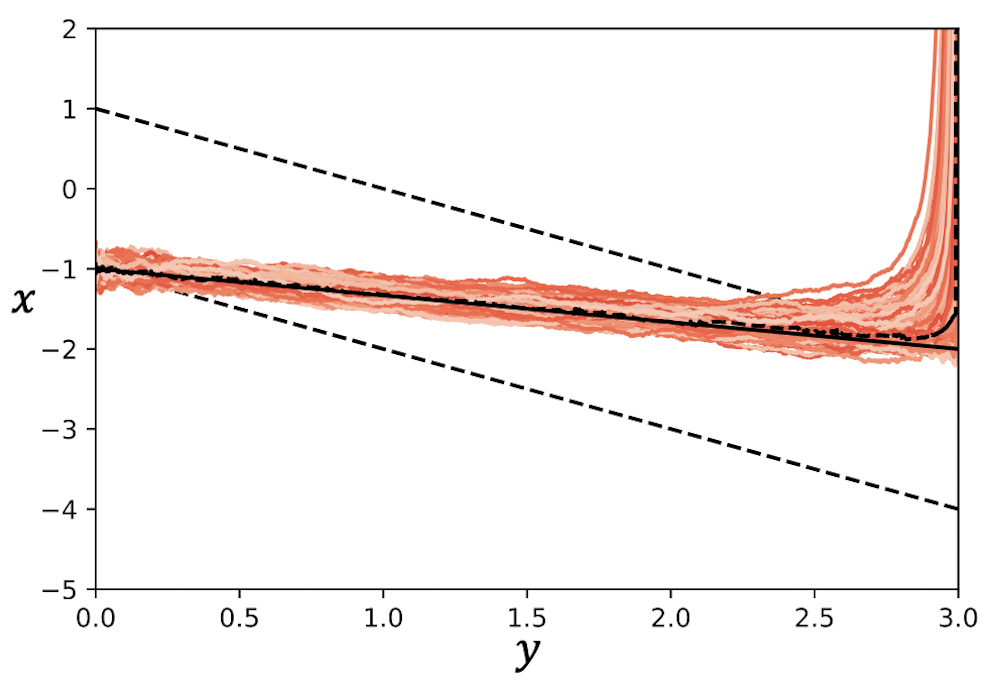}}
   \hspace{5mm}
  \subfloat[]{\includegraphics[width=0.3\textwidth]{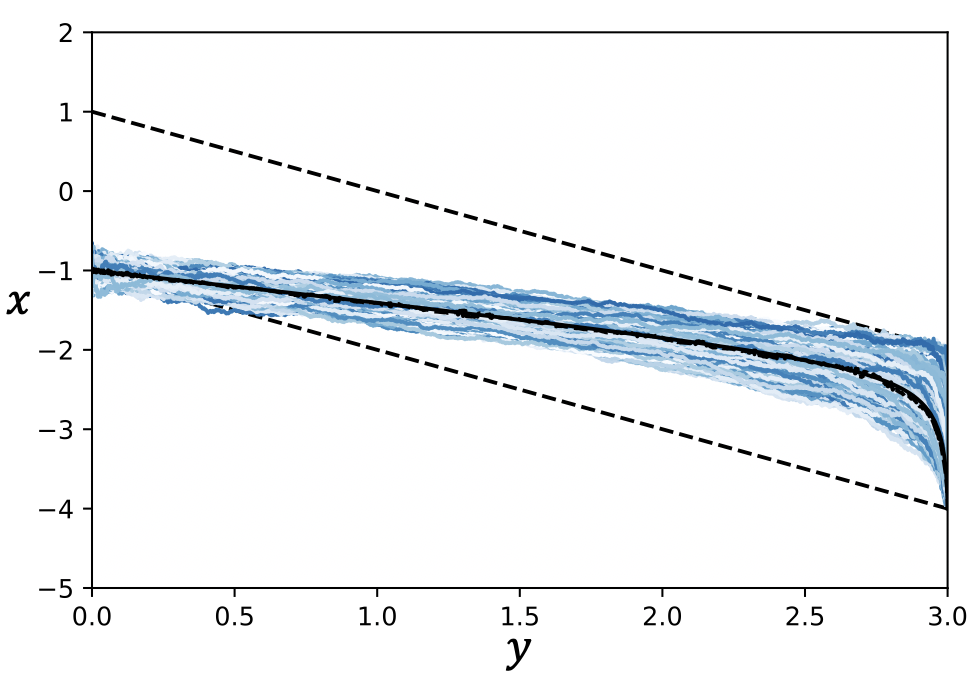}}
    \caption{a) 3000 Monte Carlo simulations of \eqref{EQ:MCs} with $r=1$ and $\sigma=.15$ on the interval $[0,30]$ with $dt=.001$. 2807 realizations do not tip (blue) and 193 tip (red). b) The realizations that tipped, overlaid with the heteroclinic orbit found (solid black) and a kernel density estimate of the realizations that tipped (dashed black). c) The realizations that did not tip, overlaid with the pullback attractor of $(-1,0)$ (solid black) and the kernel density estimate (dashed black). }
    \label{FIG:MCs}
\end{figure*}

As mentioned above, we see via the Monte Carlo simulations that trajectories either tip to infinity or end-point track the stable path to $(-4,3)$ on the given finite time horizon. The trajectories that end-point track the stable path follow the pullback attractor \cite{ashwin_parameter_2017} of $(-1,0)$. Performing another kernel smoothing density estimation on the realizations that did not tip, we see these trajectories actually peak along this pullback attractor, an example of which is shown in Figure \ref{FIG:MCs}c for $r=1, \sigma_1=0.15$.

The heteroclinic orbit and the pullback attractor are objects that can be used to separate trajectories of the system. These computations show that the addition of noise allows the system to tip when its deterministic equivalent would not tip, as the trajectory would be the pullback attractor. For the specific parameter regime $r=1,\sigma=.15$ as depicted above, even with $r$ being 3/4 of the critical rate, we are able to get tipping within the system.

\subsection{Time to Tip}


There is concern that the influence of noise on \eqref{EQ:xdotOG} is the sole reason the system exhibits tipping. However, we verify in this section that the tipping occurs due to the interplay of both the ramp parameter and noise. Additionally, the frequency of tipping largely increases with this addition of small noise strengths interacting with the ramp parameter.

Recall the original goal is to tip from $(-1,0)$ to $(-2,3)$, which correspond to $s_1$ and $s_2$ in the three-dimensional system in \eqref{EQ:MPPequations}. Using \eqref{EQ:MPPequations}, based on the stable and unstable directions of these saddles, and the direction of the vector field, we initially had two possible ways to tip due to the Hamiltonian structure in the invariant planes $y=0$ and $y=3$. We proved in Section \ref{Quadratic} that we had a third way of tipping via a heteroclinic orbit between the two saddles, which we now know is the most probable path from the Monte Carlo simulations. Refer to Figure \ref{fig:waystotip} for a visual of these possible tipping paths.

\begin{figure}
  \centering
  \subfloat[]{\includegraphics[width=.25\textwidth]{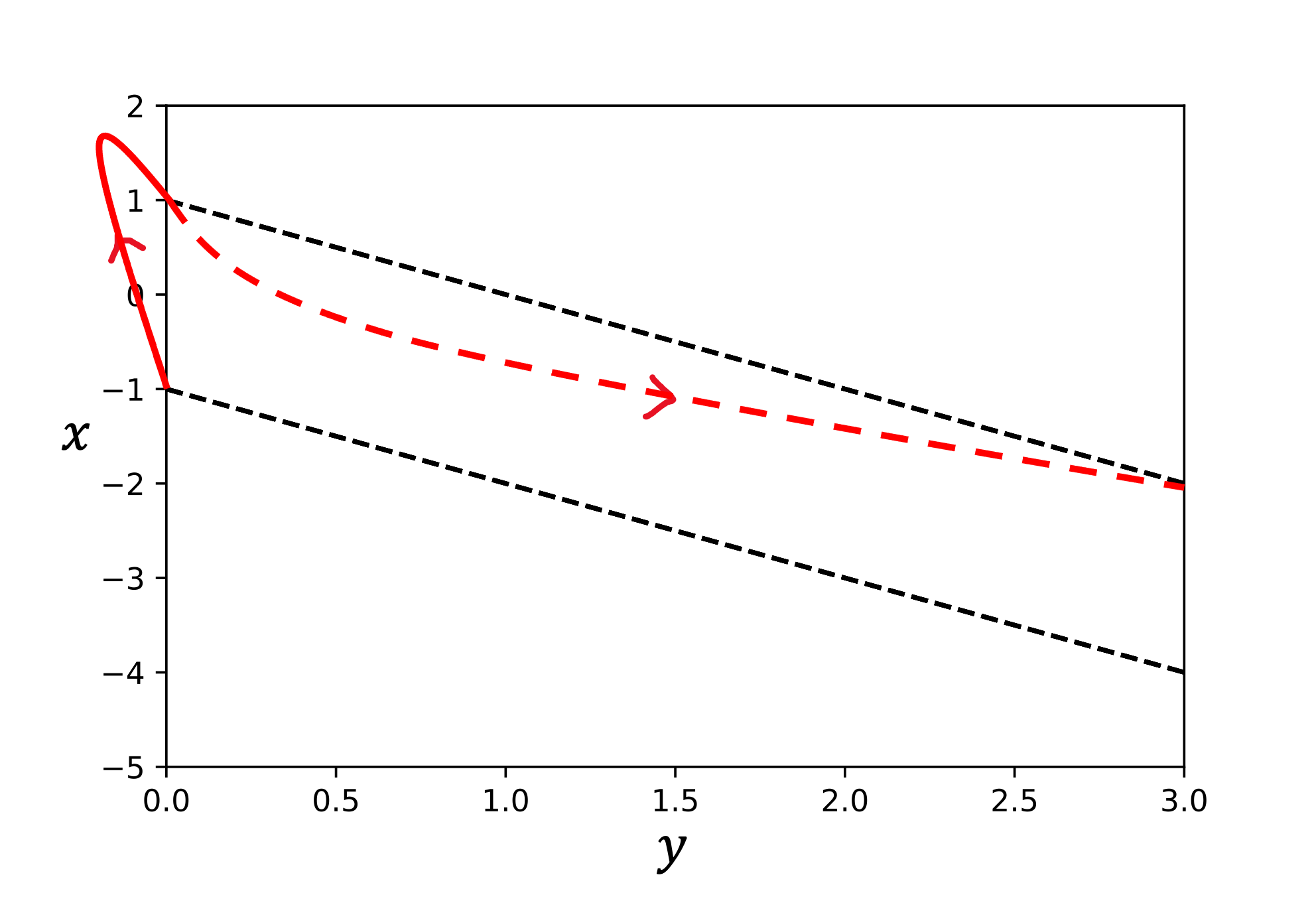}}
  \hspace{5mm}
  \subfloat[]{\includegraphics[width=.25\textwidth]{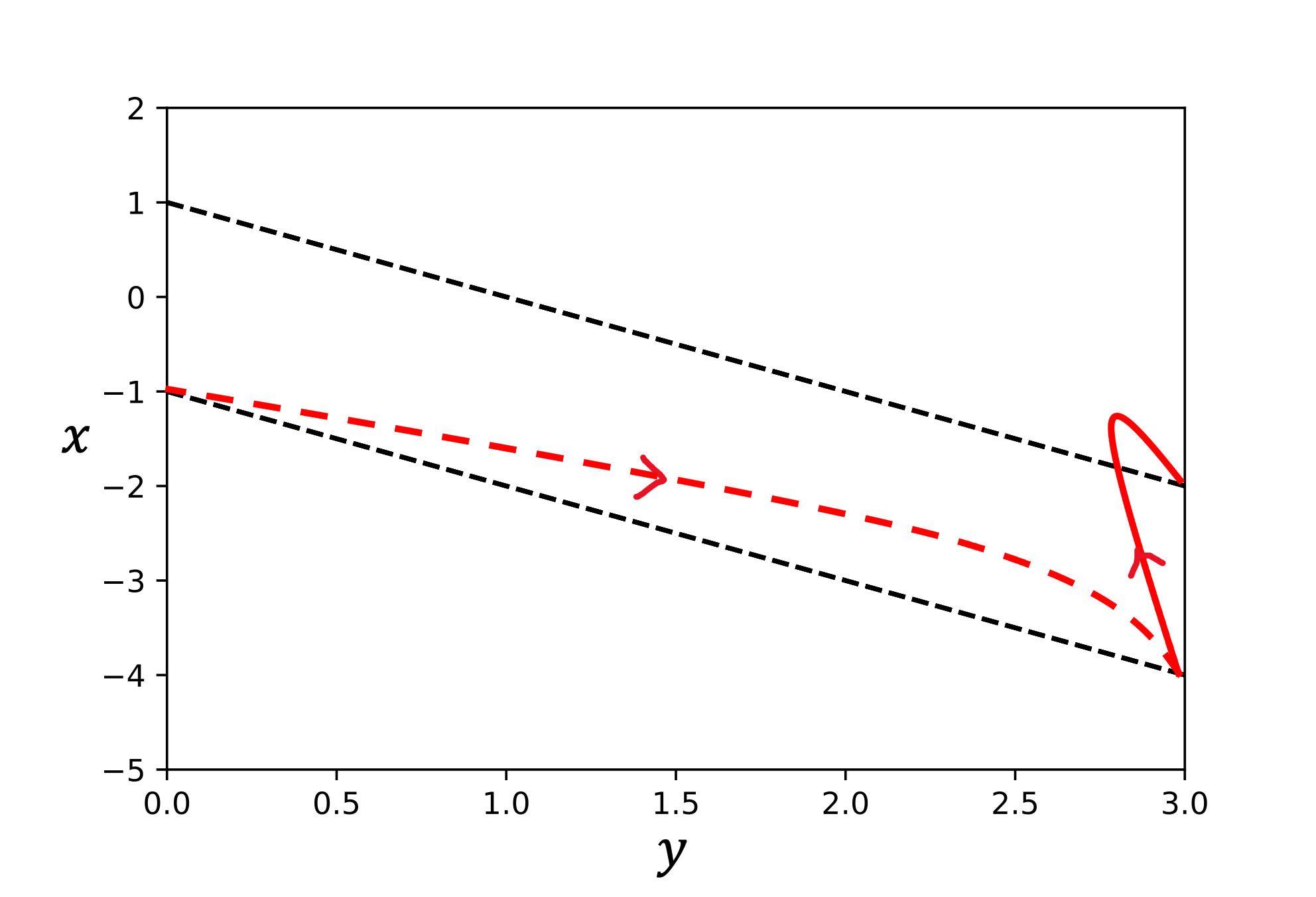}}
   \hspace{5mm}
  \subfloat[]{\includegraphics[width=.25\textwidth]{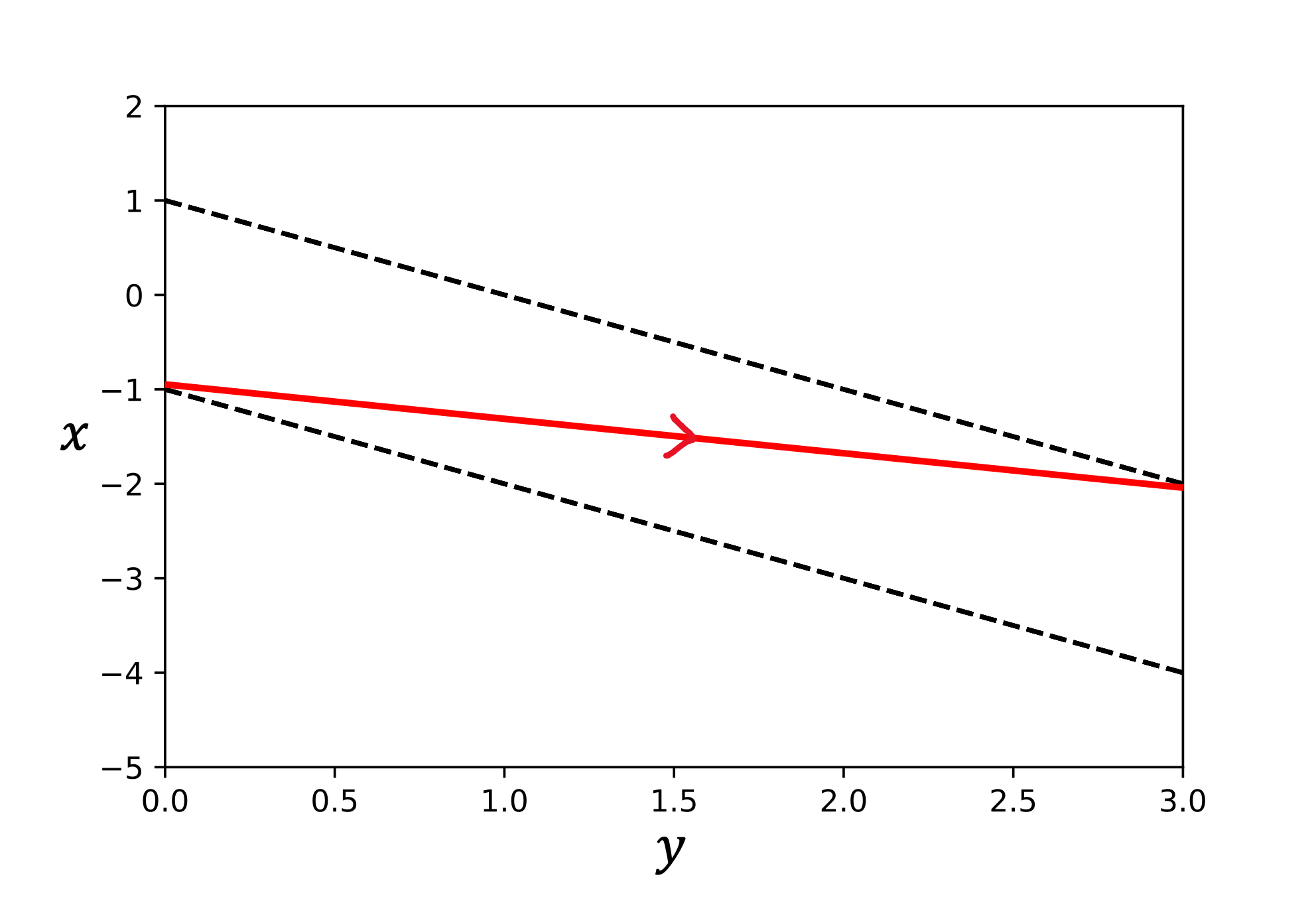}}
     \caption{The three possible ways to tip from $(-1,0)$ to $(-2,3)$. a) Case 1. b) Case 2. c) Case 3. In Cases 1 and 2, the end-point tracking curves (dashed red) will be dependent on $W^u(-1,0)$ and $W^s(-2,3)$.}
    \label{fig:waystotip}
\end{figure}

Notice that if tipping occurs in either $y=0$ or $y=3$, there is no interplay with the ramp parameter, as it would be before or after the ramping occurs. In planes $y=0,3$, we have a one-dimensional stochastic differential equation. We can find the approximate the expected time to tip as we have asymptotic formulas for gradient systems \cite{berglund_kramers_2013}, given by
\begin{equation}
\mathbb{E}[\tau] \approx e^{\frac{2 \Delta V}{\sigma^2}}.
\end{equation}
We calculate the expected time to tip for each tipping path. In our numerical analysis, $.08 \leq \sigma_1 \leq .3$ for all experiments, as we consider a small noise regime. \\ \\
Case 1. Assume we tip from $s_1$ to $(1,0,0)$ in $y=0$ and end-point track from $(1,0,0)$ to $s_2$ along $W^s(-2,3)$. We can find the expected time to tip between the fixed points in $y=0$ as the system is a gradient system in this plane. The associated form of the gradient system and potential function, $V$, is
\begin{equation}
\begin{aligned}
dx &=(x^2-1)dt + \sigma dW \\
&= -\nabla V dt + \sigma dW\\
&= -\nabla (\frac{1}{3}x^3-x)dt +\sigma dW.
\end{aligned}
\label{EQ:pot1}
\end{equation}
The extrema of $V$ correspond to the fixed points of the problem when $y=0: x=-1,1$. Solving for the expected time to tip, we find that without the ramp parameter,
\begin{equation}
\displaystyle \mathbb{E}[\tau] \approx e^{\frac{8}{3\sigma_1^2}} > 10^{12}.
\end{equation}
Therefore, the time to tip from $s_1$ to $s_2$ along this path will be some time greater than $10^{12}$. See a depiction of Case 1 in Figure \ref{fig:waystotip}a.\\ \\
Case 2. Similarly, assume we end-point track the path from $s_1$ to $(-4,3,0)$ along $W^u(-1,0)$ and then tip from $(-4,3,0)$ to $s_2$ in $y=3$. We first find the expected time to tip in $y=3$. Again, we have a gradient system and can rewrite the system in terms of the potential function $V$, written as 
\begin{equation}
\begin{aligned}
dx &=((x+3)^2-1)dt + \sigma dW \\
&= -\nabla(8 x + 3 x^2 + \frac{x^3}{3})dt +\sigma dW.
\end{aligned}
\label{EQ:pot2}
\end{equation}
Solving for the expected time to tip, we find that without the ramp parameter,
\begin{equation}
\displaystyle \mathbb{E}[\tau] \approx e^{\frac{8}{3\sigma_1^2}} > 10^{12}.
\end{equation}
The time to tip from $s_1$ to $s_2$ along this path will also be some time greater than $10^{12}$. See a depiction of Case 2 in Figure \ref{fig:waystotip}b. \\ 
\\
Case 3. To determine the expected time to tip of the most probable path found in Section \ref{Quadratic}, we run a sufficient number of Monte-Carlo simulations so that the expected time to tip distribution converges. 
We use the Euler-Maruyama method to simulate $N$ realizations of \eqref{EQ:MCs} on the interval $[0,30]$, initialized at $(-1,2.80729\times 10^{-13})$, with a step size of $dt=.001$. We want to find the realizations that have tipped to infinity, and capture when the mapped versions, $(y(t),x)$, have crossed $W^s(-2,3)$.

Let $\tau_i$ denote the first time a path, $X_i$ of the form $(y(t),x)$, crosses $W^s(-2,3)$. We define escape events to be the paths $X_i$ that have $\tau_i \leq 30$ and component $x\rightarrow \infty$. Assume for $N$ realizations there are $K$ escape events. We construct the distribution for the $K$ crossing times of $W^s(-2,3)$. To verify we have a converged result for the distribution of the time of escape events, we use the following process. 
\vspace{-2mm}
\begin{enumerate}
\item Bin the crossing times of the $K$ escape events by the Freedman Diaconis \cite{freedman_histogram_1981} rule. This separates the $K$ escape events into $B$ bins of equal length. \vspace{-2mm}
\item Run another $N$ realizations of \eqref{EQ:MCs} on the same time interval and with the same step size. Assume there are $J$ escape events. We bin the $J$ escape events by the same number of bins $B$ found in Step 1. \vspace{-2mm}
\item There are two vectors $D_1,D_2$ of the same length, where each component of the vector represents the amount of paths that tipped in that time interval. Calculate $Err=\frac{||D_1-D_2||_2}{||D_1||_2}$, which is the relative error between the two data sets. 
\vspace{-2mm}
\item If $Err<.1$, we say we have found the converged distribution. However, if $Err \geq .1$, we iterate this process with larger $N$ until the relative error of $D_1$ and $D_2$ is small enough. 
In addition, we use the Kolmogorov-Smirnov Two Sample Test \cite{noauthor_kolmogorovsmirnov_2008} as a final verification that we have a converged distribution.
\end{enumerate}

We conduct this experiment for different values of $r,\sigma$ pairs. In Table \ref{rsigtable}, we see ranges of some of the expected times to tip. Notice unlike Cases 1 and 2, the expected time to tip is now finite. The different times to tip between $s_1$ and $s_2$, depending on which path taken, demonstrates that tipping without the ramp is extremely rare to the point of almost never tipping. In addition, if we just had a ramp parameter and no stochastic component, there is no tipping for when $r<4/3$. Thus, there is an interplay of additive noise and a ramp parameter, and together they facilitate tipping on a finite timescale.

\begin{table}[ht]
\caption{\label{rsigtable}Monte Carlo simulation results for the expected time to tip for $r=.75,.85,1,1.1$ for different values of $\sigma_1$. These times come from converged results of the Monte Carlo simulations using the method described above.}
\begin{ruledtabular}
\begin{tabular}{ccc}
$r$ & $\sigma_1$ range & MC time to tip \\ \hline
.75 & $.15-.3$        &   $\sim 13-14 $             \\
.85 & .$1-.3 $    &  $ \sim 11.5-12.5  $           \\
1 & $.08-.25$         &   $\sim 9.7-10.5$             \\
1.1 & $.08-.25$       &   $\sim 8.6-9.5  $   
\end{tabular}
\end{ruledtabular}
\end{table}

\subsection{Path Actions}
In addition to using Monte Carlo simulations to see how realizations behave and to calculate the expected time to tip, we can compute the action along the different path options. The most probable path should be the path of least action. 
\
Due to the choice for the $p$ variable in the Legendre transform, for a fixed $r$, the variation of $\sigma_1$ results in a scaling in $p$. Therefore we want to consider the normalized action when calculating the path actions. The normalized action is given by:
\begin{equation}
    I[x]=\int_{t_0}^{t_f}(\dot{x}-f)^2 dt=\int_{t_0}^{t_f} \sigma_1^4 p^2 dt,
    \label{EQ:normaction}
    \end{equation}
Using \eqref{EQ:normaction} we find the heteroclinic constructed in Section 4, Case 3, has the least action compared to the other two paths of tipping, Cases 1 and 2. We see that if we tip before the ramp starts or after the ramp finishes, the action value is 5.333. However, tipping along the most probable path gives the least action value, by multiple orders of magnitude. Refer to Table \ref{actiontable} for the comparison of the action size for each of Cases $1-3$ for different $r$ values.

\begin{table}[ht]
\caption{\label{actiontable} Action values for Cases 1-3. Cases 1 and 2 do not depend on $r$ as they tip either before or after the ramp. For Case 3, which is dependent on $r$, we see that for different $r$ values the action is much less than the action of the other paths to tip.}
\begin{ruledtabular}
\begin{tabular}{ccc}
Case \#     & r value  & Action \\ \hline
1 & -   & 5.333  \\ 
2 & -   & 5.333  \\ 
3 & 1.1 & .023   \\
3       & 1   & .054   \\
3       & .75 & .226   \\
3       & .5  & .684 
\end{tabular}
\end{ruledtabular}
\end{table}

\section{Discussion and Conclusions}

\subsection{Scaling Law for the Expected Time to Tip}
For vanishingly small noise, Freidlin-Wentzell theory of large deviations, which gives the probability of a specific trajectory in a stochastic dynamical system, aids in finding the most probable path between two points. This is obtained by minimizing the Freidlin-Wentzell action functional. Additionaly, Freidlin-Wentzell theory gives the expected time to tip \cite{freidlin_random_2012}. 
We saw in this work that Freidlin-Wentzell theory holds in regard to the dynamical structure of the most probable path for small noise strengths. It is still an open question if the expected time to tip aligns between the vanishingly small noise case and the small noise case.

We discovered a power scaling law for the expected time to tip via Monte Carlo simulations, $\tau$, and $1/\sigma_1^2$, for set $r$ and varying values of $\sigma_1$. The log-log plot of these coordinate pairs result in a linear relationship, examples of which are shown in Figure \ref{FIG:loglog}. This linear relationship in log-log space corresponds to a power law of the form $a (\frac{1}{\sigma_1^2})^b$ between $1/\sigma_1^2$ and the expected time to tip. While the scaling laws in Figure \ref{FIG:loglog} are for $r=1$ and $r=.75$, the linear relationship in log-log space held true for multiple $r$ values we studied. An interesting observation is the slope of the line in log-log space for $r=1$ in Figure \ref{FIG:loglog}a is the same as the $\frac{1}{2}(\text{Action Value } r=1)$, the value of which can be seen in Table \ref{actiontable}.

This scaling law is different from the asymptotic formula given by Freidlin-Wentzell theory. However, there are various explanations for this mismatch. The most likely is that we are considering a small noise regime, and not $\sigma_1 \rightarrow 0$, and so it is not necessarily surprising the known scaling law does not hold. Alternatively, we have yet to find the leading coefficient, $c$, which could be dependent on $r$. We hypothesize that you can find the leading coefficient, by finding more expected times to tip and switching perspectives to that of inverse problems. 

We believe this task needs to implement importance sampling \cite{yu_importance_2019} to aid in speeding up the time required to gather the converged data sets. Importance sampling is commonly used to speed up Monte Carlo simulations of rare events by biasing realizations to those rare events \cite{forgoston_primer_2018}.

We would like to point out that \citet{ritchie_early-warning_2016} found that for rate values between $r=1.05$ and $r=1.25$ that as the noise is decreased, the time to tip increases slowly. They find a similar relationship for the delay in
the rate-induced tipping as that of \citet{bakhtin_gumbel_2013} for rare escapes of an autonomous system. 

\subsection{Final Conclusions}
Using compactification \cite{wieczorek_compactification_2021} with a  coordinate transformation of the ramp parameter \eqref{EQ:ramp} allows us to frame the canonical problem as a two-dimensional autonomous system with fixed points and invariant objects, as well as study the heteroclinic connection. We have shown the addition of additive noise causes the system to tip well below the critical rate needed for rate-induced tipping to occur. The system will always have a heteroclinic connection directly between the two saddle equilibria for all $r \leq r_c$. Moreover, the heteroclinic orbit found using the intersection of invariant manifolds matches the kernel density estimate of the noisy realizations found by Monte Carlo simulations, corroborating it as the most probable path of tipping between these two points. Calculating the action over all the possible paths between the two saddles, we find that the heteroclinic connection we constructed has the least action by multiple orders of magnitude, verifying we have truly found the most probable path between these two points. Additionally, we find that rate and noise-induced tipping conspire to facilitate tipping with increased probability, when neither tip on their own when considering a finite time horizon and $r <r_c$.

\begin{figure}
  \centering
  \subfloat[]{\includegraphics[width=.35\textwidth]{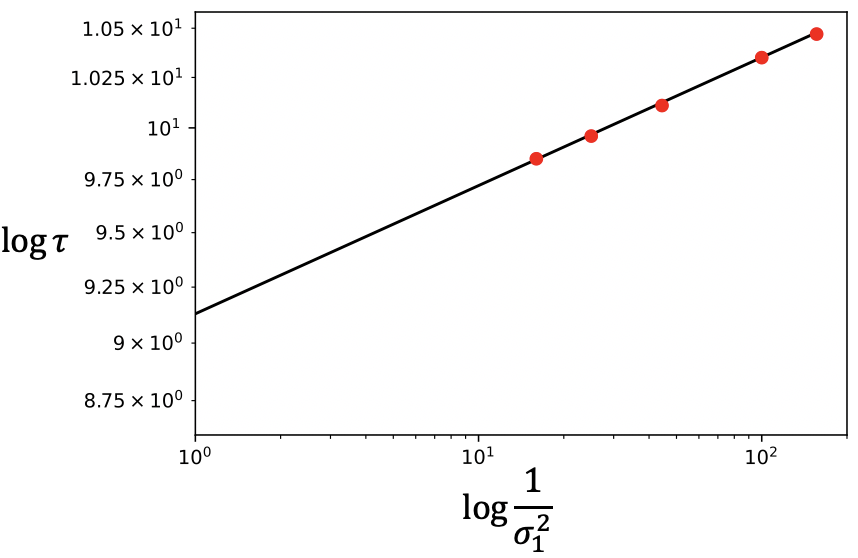}}
  \hspace{5mm}
  \subfloat[]{\includegraphics[width=.35\textwidth]{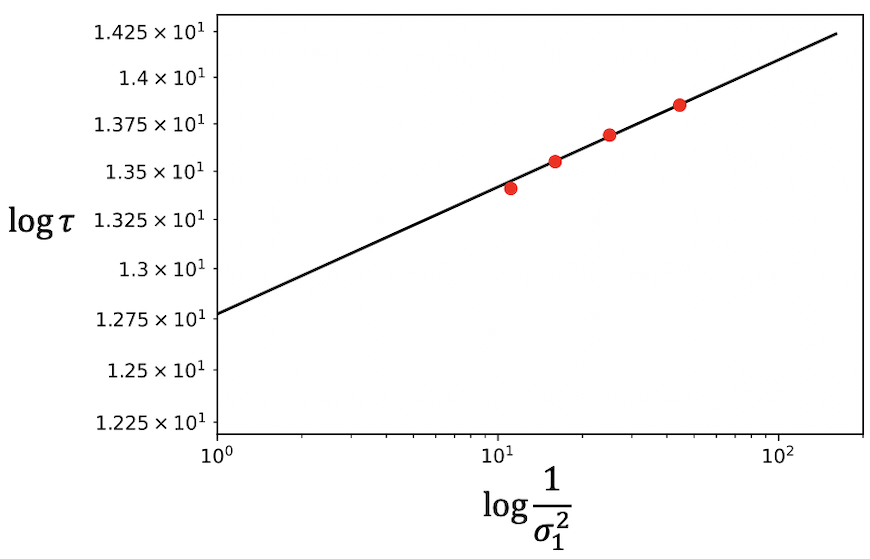}}
\caption{a) Log-log plot of $\tau$ vs. $\frac{1}{\sigma_1^2}$ for $r=1, \sigma_1=.25,.2,.15,.1,.08$. The linear relationship in log-log space corresponds to a power law of the form $9.14 (\frac{1}{\sigma_1^2})^{.027}$. \ b) Log-log plot of $\tau$ vs. $\frac{1}{\sigma_1^2}$ for $r=.75, \sigma_1=.3,.25,.2,.15$. The linear relationship in log-log space corresponds to a power law of the form $12.42 (\frac{1}{\sigma_1^2})^{0.021}$.}
    \label{FIG:loglog}
\end{figure}

We have pushed on the levels of noise to a size where Freidlin-Wentzell theory may no longer hold as the noise strength was not vanishingly small. However, we find that the Freidlin-Wentzell theory actually is still relevant in the extent of the most probable path.

This paper has considered a one-dimensional canonical problem, but we believe this work can extend to understanding tipping between a base state and threshold state of similar forms. Our method made use of the symmetry within the system. If that symmetry does not exist, other implementations of the Wazewski principle will need to be used to prove an intersection of the invariant manifolds exist. Thus, an extension to this case is still required. 

\section{Acknowledgements}
Both investigators were supported by the Office of Naval Research under grant number N000141812204.

\nocite{*}
\bibliographystyle{apsrev4-1}
\bibliography{BIBBIB}

\begin{thebibliography}{20}%
\makeatletter
\providecommand \@ifxundefined [1]{%
 \@ifx{#1\undefined}
}%
\providecommand \@ifnum [1]{%
 \ifnum #1\expandafter \@firstoftwo
 \else \expandafter \@secondoftwo
 \fi
}%
\providecommand \@ifx [1]{%
 \ifx #1\expandafter \@firstoftwo
 \else \expandafter \@secondoftwo
 \fi
}%
\providecommand \natexlab [1]{#1}%
\providecommand \enquote  [1]{``#1''}%
\providecommand \bibnamefont  [1]{#1}%
\providecommand \bibfnamefont [1]{#1}%
\providecommand \citenamefont [1]{#1}%
\providecommand \href@noop [0]{\@secondoftwo}%
\providecommand \href [0]{\begingroup \@sanitize@url \@href}%
\providecommand \@href[1]{\@@startlink{#1}\@@href}%
\providecommand \@@href[1]{\endgroup#1\@@endlink}%
\providecommand \@sanitize@url [0]{\catcode `\\12\catcode `\$12\catcode
  `\&12\catcode `\#12\catcode `\^12\catcode `\_12\catcode `\%12\relax}%
\providecommand \@@startlink[1]{}%
\providecommand \@@endlink[0]{}%
\providecommand \url  [0]{\begingroup\@sanitize@url \@url }%
\providecommand \@url [1]{\endgroup\@href {#1}{\urlprefix }}%
\providecommand \urlprefix  [0]{URL }%
\providecommand \Eprint [0]{\href }%
\providecommand \doibase [0]{http://dx.doi.org/}%
\providecommand \selectlanguage [0]{\@gobble}%
\providecommand \bibinfo  [0]{\@secondoftwo}%
\providecommand \bibfield  [0]{\@secondoftwo}%
\providecommand \translation [1]{[#1]}%
\providecommand \BibitemOpen [0]{}%
\providecommand \bibitemStop [0]{}%
\providecommand \bibitemNoStop [0]{.\EOS\space}%
\providecommand \EOS [0]{\spacefactor3000\relax}%
\providecommand \BibitemShut  [1]{\csname bibitem#1\endcsname}%
\let\auto@bib@innerbib\@empty
\bibitem [{\citenamefont {{Collins M., M. Sutherland, L. Bouwer, S.-M. Cheong,
  T. Frölicher, H. Jacot Des Combes, M. Koll Roxy, I. Losada, K. McInnes, B.
  Ratter, E. Rivera-Arriaga, R.D. Susanto, D. Swingedouw, and L.
  Tibig}}(2019)}]{collins_m_m_sutherland_l_bouwer_s-m_cheong_t_frolicher_h_jacot_des_combes_m_koll_roxy_i_losada_k_mcinnes_b_ratter_e_rivera-arriaga_rd_susanto_d_swingedouw_and_l_tibig_ocean_2022}%
  \BibitemOpen
  \bibfield  {author} {\bibinfo {author} {\bibnamefont {{Collins M., M.
  Sutherland, L. Bouwer, S.-M. Cheong, T. Frölicher, H. Jacot Des Combes, M.
  Koll Roxy, I. Losada, K. McInnes, B. Ratter, E. Rivera-Arriaga, R.D. Susanto,
  D. Swingedouw, and L. Tibig}}},\ }\enquote {\bibinfo {title} {Extremes,
  abrupt changes and managing risk},}\ in\ \href@noop {} {\emph {\bibinfo
  {booktitle} {IPCC Special Report on the Ocean and Cryosphere in a Changing
  Climate}}}\ (\bibinfo  {publisher} {Cambridge University Press},\ \bibinfo
  {year} {2019})\ pp.\ \bibinfo {pages} {589--655}\BibitemShut {NoStop}%
\bibitem [{\citenamefont {Ashwin}\ \emph {et~al.}(2012)\citenamefont {Ashwin},
  \citenamefont {Wieczorek}, \citenamefont {Vitolo},\ and\ \citenamefont
  {Cox}}]{ashwin_tipping_2012}%
  \BibitemOpen
  \bibfield  {author} {\bibinfo {author} {\bibfnamefont {P.}~\bibnamefont
  {Ashwin}}, \bibinfo {author} {\bibfnamefont {S.}~\bibnamefont {Wieczorek}},
  \bibinfo {author} {\bibfnamefont {R.}~\bibnamefont {Vitolo}}, \ and\ \bibinfo
  {author} {\bibfnamefont {P.}~\bibnamefont {Cox}},\ }\href {\doibase
  10.1098/rsta.2011.0306} {\bibfield  {journal} {\bibinfo  {journal}
  {Philosophical Transactions of the Royal Society A: Mathematical, Physical
  and Engineering Sciences}\ }\textbf {\bibinfo {volume} {370}},\ \bibinfo
  {pages} {1166} (\bibinfo {year} {2012})}\BibitemShut {NoStop}%
\bibitem [{\citenamefont {Freidlin}\ and\ \citenamefont
  {Wentzell}(2012)}]{freidlin_random_2012}%
  \BibitemOpen
  \bibfield  {author} {\bibinfo {author} {\bibfnamefont {M.~I.}\ \bibnamefont
  {Freidlin}}\ and\ \bibinfo {author} {\bibfnamefont {A.~D.}\ \bibnamefont
  {Wentzell}},\ }\href {\doibase 10.1007/978-3-642-25847-3} {\emph {\bibinfo
  {title} {Random {Perturbations} of {Dynamical} {Systems}}}},\ \bibinfo
  {series} {Grundlehren der mathematischen {Wissenschaften}}, Vol.\ \bibinfo
  {volume} {260}\ (\bibinfo  {publisher} {Springer},\ \bibinfo {address}
  {Berlin, Heidelberg},\ \bibinfo {year} {2012})\BibitemShut {NoStop}%
\bibitem [{\citenamefont {van~der Bolt}\ and\ \citenamefont {van
  Nes}(2021)}]{van_der_bolt_understanding_2021}%
  \BibitemOpen
  \bibfield  {author} {\bibinfo {author} {\bibfnamefont {B.}~\bibnamefont
  {van~der Bolt}}\ and\ \bibinfo {author} {\bibfnamefont {E.~H.}\ \bibnamefont
  {van Nes}},\ }\href {\doibase 10.1371/journal.pone.0253003} {\bibfield
  {journal} {\bibinfo  {journal} {PLOS ONE}\ }\textbf {\bibinfo {volume}
  {16}},\ \bibinfo {pages} {e0253003} (\bibinfo {year} {2021})}\BibitemShut
  {NoStop}%
\bibitem [{\citenamefont {Ashwin}\ \emph {et~al.}(2017)\citenamefont {Ashwin},
  \citenamefont {Perryman},\ and\ \citenamefont
  {Wieczorek}}]{ashwin_parameter_2017}%
  \BibitemOpen
  \bibfield  {author} {\bibinfo {author} {\bibfnamefont {P.}~\bibnamefont
  {Ashwin}}, \bibinfo {author} {\bibfnamefont {C.}~\bibnamefont {Perryman}}, \
  and\ \bibinfo {author} {\bibfnamefont {S.}~\bibnamefont {Wieczorek}},\ }\href
  {\doibase 10.1088/1361-6544/aa675b} {\bibfield  {journal} {\bibinfo
  {journal} {Nonlinearity}\ }\textbf {\bibinfo {volume} {30}},\ \bibinfo
  {pages} {2185} (\bibinfo {year} {2017})}\BibitemShut {NoStop}%
\bibitem [{\citenamefont {Ritchie}\ and\ \citenamefont
  {Sieber}(2016)}]{ritchie_early-warning_2016}%
  \BibitemOpen
  \bibfield  {author} {\bibinfo {author} {\bibfnamefont {P.}~\bibnamefont
  {Ritchie}}\ and\ \bibinfo {author} {\bibfnamefont {J.}~\bibnamefont
  {Sieber}},\ }\href {\doibase 10.1063/1.4963012} {\bibfield  {journal}
  {\bibinfo  {journal} {Chaos: An Interdisciplinary Journal of Nonlinear
  Science}\ }\textbf {\bibinfo {volume} {26}},\ \bibinfo {pages} {093116}
  (\bibinfo {year} {2016})}\BibitemShut {NoStop}%
\bibitem [{\citenamefont {Wieczorek}\ \emph {et~al.}(2021)\citenamefont
  {Wieczorek}, \citenamefont {Xie},\ and\ \citenamefont
  {Jones}}]{wieczorek_compactification_2021}%
  \BibitemOpen
  \bibfield  {author} {\bibinfo {author} {\bibfnamefont {S.}~\bibnamefont
  {Wieczorek}}, \bibinfo {author} {\bibfnamefont {C.}~\bibnamefont {Xie}}, \
  and\ \bibinfo {author} {\bibfnamefont {C.~K. R.~T.}\ \bibnamefont {Jones}},\
  }\href {\doibase 10.1088/1361-6544/abe456} {\bibfield  {journal} {\bibinfo
  {journal} {Nonlinearity}\ }\textbf {\bibinfo {volume} {34}},\ \bibinfo
  {pages} {2970} (\bibinfo {year} {2021})}\BibitemShut {NoStop}%
\bibitem [{\citenamefont {Perryman}(2015)}]{perryman_how_2015}%
  \BibitemOpen
  \bibfield  {author} {\bibinfo {author} {\bibfnamefont {C.~G.}\ \bibnamefont
  {Perryman}},\ }\href {https://ore.exeter.ac.uk/repository/handle/10871/21141}
  {\bibfield  {journal} {\bibinfo  {journal} {PhD thesis, University of
  Exeter}\ } (\bibinfo {year} {2015})}\BibitemShut {NoStop}%
\bibitem [{\citenamefont {US~EPA}(2015)}]{us_epa_climate_2015}%
  \BibitemOpen
  \bibfield  {author} {\bibinfo {author} {\bibfnamefont {O.}~\bibnamefont
  {US~EPA}},\ }\href {https://www.epa.gov/climate-indicators/greenhouse-gases}
  {\enquote {\bibinfo {title} {Climate {Change} {Indicators}: {Greenhouse}
  {Gases}},}\ } (\bibinfo {year} {2015})\BibitemShut {NoStop}%
\bibitem [{\citenamefont {Lenton}(2011)}]{lenton_early_2011}%
  \BibitemOpen
  \bibfield  {author} {\bibinfo {author} {\bibfnamefont {T.~M.}\ \bibnamefont
  {Lenton}},\ }\href {\doibase 10.1038/nclimate1143} {\bibfield  {journal}
  {\bibinfo  {journal} {Nature Climate Change}\ }\textbf {\bibinfo {volume}
  {1}},\ \bibinfo {pages} {201} (\bibinfo {year} {2011})}\BibitemShut {NoStop}%
\bibitem [{\citenamefont {Lenton}\ \emph {et~al.}(2008)\citenamefont {Lenton},
  \citenamefont {Held}, \citenamefont {Kriegler}, \citenamefont {Hall},
  \citenamefont {Lucht}, \citenamefont {Rahmstorf},\ and\ \citenamefont
  {Schellnhuber}}]{lenton_tipping_2008}%
  \BibitemOpen
  \bibfield  {author} {\bibinfo {author} {\bibfnamefont {T.~M.}\ \bibnamefont
  {Lenton}}, \bibinfo {author} {\bibfnamefont {H.}~\bibnamefont {Held}},
  \bibinfo {author} {\bibfnamefont {E.}~\bibnamefont {Kriegler}}, \bibinfo
  {author} {\bibfnamefont {J.~W.}\ \bibnamefont {Hall}}, \bibinfo {author}
  {\bibfnamefont {W.}~\bibnamefont {Lucht}}, \bibinfo {author} {\bibfnamefont
  {S.}~\bibnamefont {Rahmstorf}}, \ and\ \bibinfo {author} {\bibfnamefont
  {H.~J.}\ \bibnamefont {Schellnhuber}},\ }\href {\doibase
  10.1073/pnas.0705414105} {\bibfield  {journal} {\bibinfo  {journal}
  {Proceedings of the National Academy of Sciences}\ }\textbf {\bibinfo
  {volume} {105}},\ \bibinfo {pages} {1786} (\bibinfo {year}
  {2008})}\BibitemShut {NoStop}%
\bibitem [{\citenamefont {Forgoston}\ and\ \citenamefont
  {Moore}(2018)}]{forgoston_primer_2018}%
  \BibitemOpen
  \bibfield  {author} {\bibinfo {author} {\bibfnamefont {E.}~\bibnamefont
  {Forgoston}}\ and\ \bibinfo {author} {\bibfnamefont {R.~O.}\ \bibnamefont
  {Moore}},\ }\href {\doibase 10.1137/17M1142028} {\bibfield  {journal}
  {\bibinfo  {journal} {SIAM Review}\ }\textbf {\bibinfo {volume} {60}},\
  \bibinfo {pages} {969} (\bibinfo {year} {2018})}\BibitemShut {NoStop}%
\bibitem [{\citenamefont {Berglund}(2013)}]{berglund_kramers_2013}%
  \BibitemOpen
  \bibfield  {author} {\bibinfo {author} {\bibfnamefont {N.}~\bibnamefont
  {Berglund}},\ }\href {\doibase 10.48550/arXiv.1106.5799} {\enquote {\bibinfo
  {title} {Kramers' law: {Validity}, derivations and generalisations},}\ }
  (\bibinfo {year} {2013})\BibitemShut {NoStop}%
\bibitem [{\citenamefont {Arnold}(1997)}]{arnold_mathematical_1997}%
  \BibitemOpen
  \bibfield  {author} {\bibinfo {author} {\bibfnamefont {V.~I.}\ \bibnamefont
  {Arnold}},\ }\href@noop {} {\emph {\bibinfo {title} {Mathematical methods of
  classical mechanics}}},\ \bibinfo {edition} {2nd}\ ed.,\ \bibinfo {series}
  {Graduate texts in mathematics}\ No.~\bibinfo {number} {60}\ (\bibinfo
  {publisher} {Springer},\ \bibinfo {address} {New York},\ \bibinfo {year}
  {1997})\BibitemShut {NoStop}%
\bibitem [{\citenamefont {Srzednicki}(2004)}]{srzednicki_wazewski_2004}%
  \BibitemOpen
  \bibfield  {author} {\bibinfo {author} {\bibfnamefont {R.}~\bibnamefont
  {Srzednicki}},\ }in\ \href {\doibase 10.1016/S1874-5725(00)80009-7} {\emph
  {\bibinfo {booktitle} {Handbook of {Differential} {Equations}: {Ordinary}
  {Differential} {Equations}}}},\ Vol.~\bibinfo {volume} {1}\ (\bibinfo
  {publisher} {Elsevier},\ \bibinfo {year} {2004})\ pp.\ \bibinfo {pages}
  {591--684}\BibitemShut {NoStop}%
\bibitem [{\citenamefont {Higham.}(2001)}]{higham_algorithmic_2001}%
  \BibitemOpen
  \bibfield  {author} {\bibinfo {author} {\bibfnamefont {D.~J.}\ \bibnamefont
  {Higham.}},\ }\href {\doibase 10.1137/S0036144500378302} {\bibfield
  {journal} {\bibinfo  {journal} {SIAM Review}\ }\textbf {\bibinfo {volume}
  {43}},\ \bibinfo {pages} {525} (\bibinfo {year} {2001})}\BibitemShut
  {NoStop}%
\bibitem [{\citenamefont {Freedman}\ and\ \citenamefont
  {Diaconis}(1981)}]{freedman_histogram_1981}%
  \BibitemOpen
  \bibfield  {author} {\bibinfo {author} {\bibfnamefont {D.}~\bibnamefont
  {Freedman}}\ and\ \bibinfo {author} {\bibfnamefont {P.}~\bibnamefont
  {Diaconis}},\ }\href {\doibase 10.1007/BF01025868} {\bibfield  {journal}
  {\bibinfo  {journal} {Zeitschrift fur Wahrscheinlichkeitstheorie und
  Verwandte Gebiete}\ }\textbf {\bibinfo {volume} {57}},\ \bibinfo {pages}
  {453} (\bibinfo {year} {1981})}\BibitemShut {NoStop}%
\bibitem [{\citenamefont {Dodge}(2008)}]{noauthor_kolmogorovsmirnov_2008}%
  \BibitemOpen
  \bibfield  {author} {\bibinfo {author} {\bibfnamefont {Y.}~\bibnamefont
  {Dodge}},\ }in\ \href {\doibase 10.1007/978-0-387-32833-1_214} {\emph
  {\bibinfo {booktitle} {The {Concise} {Encyclopedia} of {Statistics}}}}\
  (\bibinfo  {publisher} {Springer},\ \bibinfo {address} {New York, NY},\
  \bibinfo {year} {2008})\ pp.\ \bibinfo {pages} {283--287}\BibitemShut
  {NoStop}%
\bibitem [{\citenamefont {Yu}\ \emph {et~al.}(2019)\citenamefont {Yu},
  \citenamefont {Muratov},\ and\ \citenamefont {Moore}}]{yu_importance_2019}%
  \BibitemOpen
  \bibfield  {author} {\bibinfo {author} {\bibfnamefont {Y.}~\bibnamefont
  {Yu}}, \bibinfo {author} {\bibfnamefont {C.~B.}\ \bibnamefont {Muratov}}, \
  and\ \bibinfo {author} {\bibfnamefont {R.~O.}\ \bibnamefont {Moore}},\ }\href
  {\doibase 10.1109/TMAG.2019.2914993} {\bibfield  {journal} {\bibinfo
  {journal} {IEEE Transactions on Magnetics}\ }\textbf {\bibinfo {volume}
  {55}},\ \bibinfo {pages} {1} (\bibinfo {year} {2019})}\BibitemShut {NoStop}%
\bibitem [{\citenamefont {Bakhtin}(2013)}]{bakhtin_gumbel_2013}%
  \BibitemOpen
  \bibfield  {author} {\bibinfo {author} {\bibfnamefont {Y.}~\bibnamefont
  {Bakhtin}},\ }\href {\doibase 10.48550/arXiv.1312.1939} {\enquote {\bibinfo
  {title} {On {Gumbel} limit for the length of reactive paths},}\ } (\bibinfo
  {year} {2013}),\ \bibinfo {note} {arXiv:1312.1939 [math]}\BibitemShut
  {NoStop}%
\end{thebibliography}%

\end{document}